%% file: tie_ff.tex
\documentclass[10pt,reqno]{amsproc}

\usepackage{fullpage}
\input{preamble}

\begin{document}


\title{Ties in Function Field Prime Races}


\author{Graeme Bates, Ryan Jesubalan, Seewoo Lee, Jane Lu, and Hyewon Shim}
\date{}


\begin{abstract}
    The function field analogue of Chebyshev's bias was first studied by Cha. In this paper, we study \emph{ties} in this race, namely collections of distinct congruence classes \(c_1, \dots, c_k \in (\bF_q[T] / m)^\times\) for which
    \[
        \pi(N; m, c_1) = \pi(N; m, c_2) = \dots = \pi(N; m, c_k)
    \]
    holds for infinitely many $N$.
    We provide infinitely many examples of \((m, c_1, \dots, c_k)\) for which the tie holds whenever $N$ satisfies certain congruence conditions.
    We give two different proofs: first, via the explicit formula for prime counts in terms of \(L\)-functions together with a matrix analogue of M\"obius inversion, where exceptional pairs of Galois-conjugate elements in the corresponding cyclotomic fields produce ties; and second, via an explicit bijection arising from the \(\GL_2(\bF_q)\)-action.
    Our examples also include characteristic 2 cases.
\end{abstract}


\maketitle

\input{1intro}
\input{2preliminary}
\input{3explicitformula}
\input{4bijection}
\input{5conclusion}


\bibliographystyle{acm} 
\bibliography{refs} 

\medskip

\newpage
\input{appendix}






\end{document}

%% file: preamble.tex
\usepackage[dvipsnames]{xcolor}

\usepackage{
  amsmath, amsthm, amssymb, mathtools, dsfont, units,          
  graphicx, wrapfig, subfig, float,                            
  listings, color, inconsolata, pythonhighlight,               
  fancyhdr,
  hyperref,
  enumerate,
  enumitem,
  framed
}   

\usepackage{hyperref, lipsum} 
\usepackage{bm}
\usepackage{subfig}

\usepackage{newpxtext, newpxmath, inconsolata}
\usepackage{amsfonts}
\usepackage{pgfplots}
\pgfplotsset{compat=1.12}
\usepackage{tkz-fct}
\usepackage{svg}
\usepackage{tikz}
\usepackage{tikz-cd}
\usepackage{lipsum}
\usepackage{rank-2-roots}
\usepackage[utf8]{inputenc}
\usepackage{multicol}
\usepackage{multirow}
\usepackage{pifont}

\hypersetup{colorlinks=true, linkcolor=red, citecolor=blue, filecolor=blue, urlcolor=blue}

\DeclareUnicodeCharacter{3BC}{$\pi$}
\DeclareUnicodeCharacter{3C0}{$\pi$}

\usepackage{booktabs}

\usepackage[bottom]{footmisc}

\allowdisplaybreaks

\usepackage[font={it,footnotesize}]{caption}

\usepackage{titlesec}
\titleformat{\section}{\large\bfseries}{\thesection\;\;\;}{0em}{}
\titleformat{\subsection}{\normalsize\bfseries\selectfont}{\thesubsection\;\;\;}{0em}{}
\titleformat{\subsubsection}{\normalsize\bfseries\selectfont}{\thesubsubsection\;\;\;}{0em}{}

\setlist[itemize]{wide=0pt, leftmargin=16pt, labelwidth=10pt, align=left}

\graphicspath{{Images/}{../Images/}}

\usepackage[mathscr]{euscript}

\newcommand{\what}[1]{\widehat{#1}}

\newcommand{\GL}{\mathrm{GL}}

\newcommand{\Tr}{\mathrm{Tr}}

\newcommand{\Hom}{\mathrm{Hom}}

\newcommand{\id}{\mathrm{id}}
\newcommand{\lc}{\mathrm{lc}}

\newcommand{\dd}{\mathrm{d}}

\newcommand{\bC}{\mathbb{C}}

\newcommand{\bF}{\mathbb{F}}

\newcommand{\bQ}{\mathbb{Q}}

\newcommand{\bZ}{\mathbb{Z}}

\newcommand{\cI}{\mathcal{I}}

\newcommand{\cL}{\mathcal{L}}

\newcommand{\scU}{\mathscr{U}}

\theoremstyle{definition}
\newtheorem{theorem}{Theorem}[section]
\newtheorem{proposition}[theorem]{Proposition}
\newtheorem{lemma}[theorem]{Lemma}
\newtheorem{corollary}[theorem]{Corollary}

\newtheorem{conjecture}[theorem]{Conjecture}

\theoremstyle{remark}
\newtheorem{remark}[theorem]{Remark}

\newtheorem*{theorem*}{Theorem}
\newtheorem*{proposition*}{Proposition}
\newtheorem*{lemma*}{Lemma}
\newtheorem*{corollary*}{Corollary}
\newtheorem*{definition*}{Definition}
\newtheorem*{example*}{Example}
\newtheorem*{remark*}{Remark}

\newtheorem*{conjecture*}{Conjecture}

\usepackage{fancyhdr}
\newenvironment{red}{\relax\color{red}}{\hspace*{.5ex}\relax}
\newenvironment{blue}{\relax\color{blue}}{\hspace*{.5ex}\relax}
\newcommand{\ber}{\begin{red}}
\newcommand{\er}{\end{red}}
\newcommand{\beb}{\begin{blue}}
\newcommand{\eb}{\end{blue}}

\definecolor{myblue}{RGB}{0, 100, 255}
\definecolor{myred}{RGB}{220, 50, 47}
\definecolor{mygreen}{RGB}{0,150,0}
\newcommand{\myg}[1]{\textcolor{mygreen}{#1}}
\newcommand{\myb}[1]{\textcolor{myblue}{#1}}
\newcommand{\myr}[1]{\textcolor{myred}{#1}}

%% file: 1intro.tex
\section{Introduction}
\label{sec:intro}

\begin{table}[h!]
\centering
\begin{tabular}{c|cccccccc}
\toprule
\textbf{$N$} & \textbf{$1$} & \textbf{$T$} & \textbf{$T^2$} & \textbf{$T+1$} & \textbf{$T^2+T$} & \textbf{$T^2+T+1$} & \textbf{$T^2+1$} \\
\midrule
9  & \myg{7}     & \myb{9}     & \myg{7}     & \myb{9}     & \myb{9}     & {8}     & \myg{7}     \\
10 & \myg{15}    & \myb{14}    & \myg{15}    & \myg{15}    & {12}    & \myb{14}    & \myb{14}    \\
11 & \myg{28}    & \myb{26}    & \myb{26}   & {24}   & \myg{28}     & \myg{28}     & \myb{26}    \\
12 & \myg{46}    & \myg{46}     & {50}   & \myb{49}   & \myb{49}    & \myg{46}     & \myb{49}    \\
13 & \myb{89}   & {96}    & \myb{89}   & \myb{89}   & \myb{89}    & \myb{89}    & \myb{89}    \\
14 & {168}  & \myg{162}    & \myg{162}   & \myb{169}  & \myg{162}    & \myb{169}   & \myb{169}   \\
15 & \myg{310}   & \myg{310}    & \myb{316}  & \myg{310}   & \myb{316}   & \myb{316}   & {304}   \\
16 & \myg{588}   & \myb{582}   & \myg{588}   & \myb{582}  & \myb{582}   & {570}   & \myg{588}    \\
17 & \myg{1093}  & \myb{1109}  & \myg{1093}  & \myg{1093}  & {1104}  & \myb{1109}  & \myb{1109}  \\
18 & \myg{2075}  & \myb{2069}  & \myb{2069} & {2100} & \myg{2075}   & \myg{2075}   & \myb{2069}  \\
19 & \myg{3951}  & \myg{3951}   & {3960} & \myb{3927} & \myb{3927}  & \myg{3951}   & \myb{3927}  \\
20 & \myg{7502}  & {7458}  & \myb{7471} & \myb{7471} & \myg{7502}   & \myb{7471}  & \myg{7502}   \\
21 & {14208}& \myg{14282}  & \myg{14282} & \myb{14268}& \myg{14282}  & \myb{14268} & \myb{14268} \\
22 & \myg{27258} & \myg{27258}  & \myb{27189}& \myg{27258} & \myb{27189} & \myb{27189} & {27216} \\
\vdots & \vdots & \vdots & \vdots & \vdots & \vdots & \vdots & \vdots \\
\midrule
$0 \pmod{7}$ &   & \myg{\ding{73}}    & \myg{\ding{73}}   & \myb{$\circ$}  & \myg{\ding{73}}    & \myb{$\circ$}   & \myb{$\circ$}   \\
$1 \pmod{7}$ & \myg{\ding{73}}   & \myg{\ding{73}}    & \myb{$\circ$}  & \myg{\ding{73}}   & \myb{$\circ$}   & \myb{$\circ$}   &    \\
$2 \pmod{7}$ & \myg{\ding{73}}   & \myb{$\circ$}   & \myg{\ding{73}}   & \myb{$\circ$}  & \myb{$\circ$}   &    & \myg{\ding{73}}    \\
$3 \pmod{7}$ & \myg{\ding{73}}  & \myb{$\circ$}  & \myg{\ding{73}}  & \myg{\ding{73}}  &   & \myb{$\circ$}  & \myb{$\circ$}  \\
$4 \pmod{7}$ & \myg{\ding{73}}  & \myb{$\circ$}  & \myb{$\circ$} &  & \myg{\ding{73}}   & \myg{\ding{73}}   & \myb{$\circ$}  \\
$5 \pmod{7}$ & \myg{\ding{73}}  & \myg{\ding{73}}   &  & \myb{$\circ$} & \myb{$\circ$}  & \myg{\ding{73}}   & \myb{$\circ$}  \\
$6 \pmod{7}$ & \myg{\ding{73}}  &   & \myb{$\circ$} & \myb{$\circ$} & \myg{\ding{73}}   & \myb{$\circ$}  & \myg{\ding{73}}   \\
\bottomrule
\end{tabular}%
\caption{Number of irreducible polynomials modulo \(T^3 + T + 1 \in \mathbb{F}_2[T]\) in each of the seven congruence classes for each degree $N$, together with the resulting patterns modulo $7$. The columns are ordered as \(1, T, T^2, T^3, \dots, T^6\).}
\label{tab:T3T1}
\end{table}

In 1853, Chebyshev \cite{chebyshev1853lettre} observed that there are typically more primes of the form $4n + 3$ than of the form $4n + 1$ up to certain thresholds, and he conjectured that the count is essentially \emph{biased} toward primes $\equiv 3 \pmod{4}$.
More precisely, if we denote by $\pi(x; m, a)$ the number of primes $\le x$ that are congruent to $a$ modulo $m$, then he conjectured that the difference
\[
\Delta(x) := \pi(x; 4, 3) - \pi(x; 4, 1)
\]
is positive for all sufficiently large $x$.
However, Littlewood \cite{littlewood1914distribution} proved that there are infinitely many sign changes of $\Delta(x)$; in particular, there are infinitely many $x$ with $\Delta(x) < 0$.
Later, under the Linear Independence conjecture (LI) for Dirichlet $L$-functions, Rubinstein and Sarnak \cite{rubinstein1994chebyshev} gave a more precise analysis, showing that the relevant \emph{logarithmic density} is about $0.9959$, which is close to 1 but not equal to it.
Subsequent work studied such \emph{prime number races} in more general settings; see \cite{granville2006prime}.

Cha \cite{cha2008chebyshev} studied Chebyshev's bias over function fields.
Following Rubinstein and Sarnak's argument, he proved that a similar bias in prime counts toward non-quadratic residues exists under the Grand Simplicity Hypothesis (GSH) for Dirichlet $L$-functions.
While LI is believed to be true over number fields, Cha found several examples of Dirichlet $L$-functions over function fields that \emph{do not} satisfy GSH and that also exhibit bias in unexpected directions.

In this paper, we study \emph{ties} in Chebyshev's bias over function fields.
Let $m \in \bF_q[T]$ be a polynomial and $a \in (\bF_q[T] / m)^\times$ be a congruence class modulo $m$.
For $N \ge 1$, let $\pi(N; m, a)$ be the number of irreducible monic polynomials in $\bF_q[T]$ of degree $N$ which are congruent to $a$ modulo $m$.
We exhibit a family of examples of moduli $m$ and distinct congruence classes $c_1, \dots, c_k \in (\bF_q[T] / m)^\times$ for which
\[
\pi(N; m, c_1) = \pi(N; m, c_2) = \dots = \pi(N; m, c_k)
\]
holds for all $N$ in certain arithmetic progressions.
For example, the following theorem holds for \(m = T^3 + T + 1 \in \bF_2[T]\).
\begin{theorem}
    \label{thm:p2T3T1ex}
    Let \(m = T^3 + T + 1 \in \bF_2[T]\).
    For all $N \equiv 1 \pmod{7}$, we have
    \begin{equation}
    \label{eqn:tieN1mod7} 
        \pi(N; m, 1) = \pi(N; m, T) = \pi(N; m, T + 1)\quad\text{and}\quad \pi(N; m, T^2) = \pi(N; m, T^2 + T) = \pi(N; m, T^2 + T + 1).
    \end{equation}
\end{theorem}
Table \ref{tab:T3T1} shows these ties, along with the analogous patterns for the other residue classes modulo $7$.

To prove such ties, we use two different approaches.
The first method uses explicit formulas for the prime counts $\pi(N; m, a)$ in terms of zeros of Dirichlet $L$-functions and a matrix version of M\"obius inversion, where the ties arise from ``exceptional'' Galois conjugates (Section \ref{sec:explicit}).
The second method constructs bijections between primes in different congruence classes via \(\GL_2(\bF_q)\)-actions (Section \ref{sec:bijection}).
The accompanying Sage code used to implement these explicit formulas and generate the tables can be found in the GitHub repository \href{https://github.com/seewoo5/sage-function-field}{github.com/seewoo5/sage-function-field}.

\subsection*{Acknowledgements}

This work was initiated during the 2025 Berkeley Math REU program and supported by NSF RTG grant DMS-2342225, MPS Scholars, and the Gordon and Betty Moore Foundation.
We thank Tony Feng for organizing the REU, and we also thank Alexandre Bailleul for helpful comments.

%% file: 2preliminary.tex
\section{Preliminaries}
\label{sec:prelim}

\subsection{Arithmetic of function fields}
\label{subsec:prelim_arith}

Let $p$ be a prime and let $A = \bF_q[T]$ be the polynomial ring over a finite field $\bF_q$ of size $q = p^k$.
We write $A^\circ \subset A$ for the set of monic polynomials, and let $\cI \subset A^\circ$ be the set of monic irreducible polynomials.
For each $N$, we write $\cI(N)$ for the set of monic irreducible polynomials of degree $N$, and $\pi(N) = \# \cI(N)$ for the number of such polynomials.
The famous formula due to Gauss tells us that
\begin{equation}
\label{eqn:gauss_cnt}    
\pi(N) = \frac{1}{N} \sum_{d \mid N} \mu(d) q^{N / d} = \frac{q^N}{N} + O\left(\frac{q^{N/2}}{N}\right)
\end{equation}
where $\mu$ is the M\"obius function; this is the function field analogue of the prime number theorem.
When $m \in A^\circ$ is a monic polynomial and $c \in A$ is another polynomial coprime to $m$, we write $\cI(N; m, c)$ for the set of monic irreducible polynomials of degree $N$ that are congruent to $c$ modulo $m$.
Then $\pi(N; m, c) := \# \cI(N; m, c)$ is known to be
\begin{equation}
\label{eqn:dirichlet_cnt}    
\pi(N; m, c) = \frac{1}{\Phi(m)} \frac{q^N}{N} + O\left(\frac{q^{N/2}}{N}\right)
\end{equation}
as $N \to \infty$, where $\Phi(m) = \# (A / m)^\times$ is the Euler totient function for $A$; this is the function field analogue of Dirichlet's theorem on arithmetic progressions.

For a Dirichlet character $\chi$ of modulus $m \in A^\circ$, we define the corresponding $L$-function as
\begin{equation}
\label{eqn:Lchi_euler}
    L(s, \chi) :=  \sum_{n \ge 0} \left(\sum_{\substack{f \in A^\circ \\ \deg f = n}}\chi(f)\right) q^{-ns} = \prod_{\substack{P \in \cI \\ P \nmid m}} (1 - \chi(P)q^{-(\deg P)s})^{-1}
\end{equation}
which is known to be a polynomial in $u = q^{-s}$ of degree at most $M - 1$, where $M = \deg m$.
We denote the corresponding polynomial as $\cL(u, \chi) \in \bC[u]$.
Weil's Riemann Hypothesis for function fields shows that the zeros of $L(s, \chi)$ satisfy $s = 0$ or $\Re s = \frac{1}{2}$.

\subsection{Chebyshev's bias}
\label{subsec:prelim_chebyshev}

Chebyshev's bias \cite{chebyshev1853lettre} refers to the phenomenon that the number of primes of the form $4k + 3$ is usually larger than the number of those of the form $4k + 1$.
This generalizes to arbitrary modulus, where primes are biased against quadratic residues.
Major progress on the conjecture was made by Rubinstein and Sarnak \cite{rubinstein1994chebyshev}, who proved that such a bias exists for any modulus under the Linear Independence hypothesis (LI) for Dirichlet $L$-functions, which asserts that all nonnegative ordinates of Dirichlet $L$-functions are linearly independent over $\bQ$.

Some earlier works \cite{ingham1935note,knapowski1961sign,pintzsign,martin2020inclusive} also studied the number of \emph{sign changes} or \emph{ties} in the bias, and every sign change gives an obvious lower bound for the number of ties in the same prime race.
Although \cite{ingham1935note,pintzsign,knapowski1961sign} mostly focuses on sign changes of $V(x) := \pi(x) - \mathrm{li}(x)$, the same arguments generalize to the mod $4$ race $\Delta(x)$.
Ingham \cite{ingham1935note} proved that $V(x) = \Omega(\log x)$, assuming that the supremum of the real parts of the zeros of $\zeta(s)$ is attained by a zero, which is weaker than RH.
An unconditional (ineffective) lower bound of $V(x) = \Omega(\log \log x)$ was first obtained by Knapowski \cite{knapowski1961sign}, while Pintz \cite{pintzsign} proved the effective lower bound $V(x) = \Omega(\sqrt{\log x} / \log \log x)$.
Martin and Ng \cite{martin2020inclusive} studied prime races under hypotheses weaker than LI and proved that the ties have logarithmic density zero.

Cha \cite{cha2008chebyshev} studied the function field analogue of Chebyshev's bias. Following Rubinstein and Sarnak's argument, he proved analogous results for polynomials over finite fields of odd characteristic.
More precisely, for a monic polynomial $m \in \bF_q[T]$, if GSH holds for Dirichlet $L$-functions of Dirichlet characters modulo $m$, he proved that the natural density of degrees $N$ for which there are more non-quadratic residues of degree $N$ than quadratic residues modulo $m$ is larger than $\frac{1}{2}$.
Surprisingly, there are examples in which the Grand Simplicity Hypothesis (GSH, the function field analogue of LI for $L$-functions) fails, and he also gave examples in which the bias points in an unexpected direction, such as toward quadratic residues.
He also proved that the bias vanishes as $\deg m \to \infty$.

Bailleul, Devin, Keliher, and Li studied exceptional biases in function fields \cite{bailleul2024exceptional}, such as \emph{complete bias}, \emph{lower-order bias}, and \emph{reversed bias}.
In particular, when $m$ is a square-free monic polynomial and $\chi_m$ is the unique principal quadratic character modulo $m$ (so necessarily $q$ is odd), they gave necessary and sufficient conditions for $\chi_m$ to admit lower-order bias, that is, for the set of degrees $N$ with ties
\begin{equation}
\begin{aligned}
    &\# \{ f \in \bF_q[T] : f\text{ monic, irreducible, } \deg f = N, \chi_m(f) = 1\} \\
    &= \# \{ f \in \bF_q[T] : f\text{ monic, irreducible, } \deg f = N, \chi_m(f) = -1\} 
\end{aligned}
\end{equation}
has positive natural density.
For example, they proved that there is lower-order bias for $\chi_m$ when the completed $L$-polynomial for $\chi_m$ is even (see \cite[Lemma 5.2]{bailleul2024exceptional} for the precise statement).

%% file: 3explicitformula.tex
\section{Explicit counting formulas}
\label{sec:explicit}

\subsection{Matrix M\"obius inversion and explicit counting formulas}

As mentioned earlier in Section \ref{subsec:prelim_arith}, one can prove \eqref{eqn:dirichlet_cnt} by considering all Dirichlet $L$-functions together and using Weil's RH to estimate the error terms.
Here we compute all terms more explicitly, without introducing error bounds, and obtain exact formulas for $\pi(N; m, c)$ as a M\"obius inversion of matrices associated with the modulus $m$.

Let \(M = \deg m\). 
We denote the unit group \((A / m)^\times\) and its character group as \(\scU_m := (A / m)^\times\) and \(\what{\scU}_m := \Hom(\scU_m, \bC^\times)\), respectively.
Let $M' = \Phi(m) = |\scU_m|$ be the size of the group.
For \(\chi \in \what{\scU}_m\) and $n \ge 1$, define
\begin{equation}
\label{eqn:Achi}
A_\chi(n) := \sum_{c \in \scU_m} \pi(n; m, c) \chi(c).
\end{equation}
For nontrivial $\chi$, write $\alpha_1(\chi), \dots, \alpha_d(\chi)$ with $d = d(\chi)$ for the inverse zeros of $\cL(u, \chi)$, i.e.
\begin{equation}
    \cL(u, \chi) = \prod_{j=1}^{d(\chi)} (1 - \alpha_j(\chi) u).
\end{equation}
By taking the logarithmic derivative, one easily sees that
\begin{equation}
\label{eqn:Llogder}
    u \frac{\dd}{\dd u} \log \cL(u, \chi) = \frac{u \cL'(u, \chi)}{\cL(u, \chi)} = -\sum_{n \ge 1} \left(\sum_{j=1}^{d(\chi)}\alpha_j(\chi)^n\right) u^n =: \sum_{n \ge 1} c_n(\chi) u^n.
\end{equation}
where
\begin{equation}
\label{eqn:cn}
    c_n(\chi) := -\sum_{j=1}^{d(\chi)} \alpha_j(\chi)^n.
\end{equation}
Taking the logarithmic derivative of the Euler product expansion \eqref{eqn:Lchi_euler} gives
\begin{align}
    u \frac{\dd}{\dd u} \log \cL(u, \chi) &= \sum_{P \in \cI} \frac{\chi(P) \deg P u^{\deg P}}{1 - \chi(P) u^{\deg P}} = \sum_{P \in \cI} \deg P\sum_{k \ge 1} \chi(P)^k u^{k \deg P} \nonumber \\
    &= \sum_{n \ge 1} \sum_{d \mid n} d \sum_{P \in \cI(d)} \chi(P)^{n/d} u^n \nonumber \\
    &= \sum_{n \ge 1} \sum_{d \mid n} d \sum_{c \in \scU_m} \pi(d; m, c) \chi^{n/d}(c) u^n \nonumber \\
    &= \sum_{n \ge 1} \left(\sum_{d \mid n} d A_{\chi^{n/d}}(d)\right) u^n\label{eqn:cn2}
\end{align}
and comparing \eqref{eqn:cn} and \eqref{eqn:cn2} gives
\begin{equation}
\label{eqn:dAsum}
    c_n(\chi) = \sum_{d \mid n} d A_{\chi^{n/d}}(d).
\end{equation}
When $\chi = \chi_0$ is the trivial character modulo $m$, we have
\begin{equation}
    A_{\chi_0}(d) = \sum_{c \in \scU_m} \pi(d; m, c) = \pi(d) - \sum_{\substack{P \mid m\,\text{irred} \\ \deg P = d}} 1
\end{equation}
and summing $d A_{\chi_0}(d)$ over $d \mid n$ gives
\begin{equation}
\label{eqn:dAtrivsum}
    \sum_{d \mid n} dA_{\chi_0}(d) = \sum_{d \mid n} d \pi(d) - \sum_{d \mid n} d\sum_{\substack{P \mid m\,\text{irred} \\ \deg P = d}} 1 = q^n - \sum_{\substack{P \mid m\,\text{irred} \\ \deg P \mid n}} \deg P,
\end{equation}
where the last equality follows from \cite[Proposition 2.1]{rosen2013number}.
We denote
\begin{equation}
\label{eqn:sm}
    s_{m,n} := \sum_{\substack{P \mid m\,\text{irred} \\ \deg P \mid n}} \deg P
\end{equation}
so that the last term of \eqref{eqn:dAtrivsum} becomes $q^n - s_{m, n}$.
When $m$ is irreducible, then $s_{m, n} = M \delta_{M \mid n}$ where $\delta_{M \mid n}$ is 1 (resp. 0) if $M \mid n$ (resp. $M \nmid n$).

Now, define a map \(Z(n) \in \Hom(\bC^{\scU_m}, \bC^{\what{\scU}_m})\) as
\begin{equation}
    v = (v_a)_{a \in \scU_m} \mapsto v' = (v'_\chi)_{\chi \in \what{\scU}_m}, \,\, v'_\chi = \sum_{a\in \scU_m} \chi^n(a) v_a
\end{equation}
for each \(n \in \bZ\); this can be viewed as a matrix whose \((\chi, a)\)-th entry is \(\chi^n(a)\).
Then \eqref{eqn:dAsum} and \eqref{eqn:dAtrivsum} can be encoded as a system of linear equations
\begin{equation}
\label{eqn:coeffsystem1}
    \begin{pmatrix}
        q^n - s_{m, n} \\
        \vdots \\
        c_n(\chi) \\
        \vdots
    \end{pmatrix} = \sum_{d \mid n} d \begin{pmatrix}
        A_{\chi_0}(d) \\ \vdots \\ A_\chi(d) \\ \vdots
    \end{pmatrix} = \sum_{d \mid n} Z\left(\frac{n}{d}\right) \begin{pmatrix}
        d \pi(d; m, 1) \\ \vdots \\ d \pi(d; m, a) \\ \vdots
    \end{pmatrix}.
\end{equation}
To obtain a formula for \(\pi(N; m, a)\) by solving this system of equations, we prove a M\"obius inversion formula for \(Z(n)\).
For each \(n \ge 1\), we define \emph{M\"obius inverse} \(\widetilde{Z}(n) \in \Hom(\bC^{\what{\scU}_m}, \bC^{\scU_m})\) inductively as
\begin{align}
    \widetilde{Z}(1) Z(1) &= \id_{\bC^{\scU_m}} \label{eqn:Ztilde_one} \\
    \sum_{d \mid n} \widetilde{Z}(d) Z\left(\frac{n}{d}\right) &= 0 \Leftrightarrow \widetilde{Z}(n) Z(1) = -\sum_{d \mid n, d < n} \widetilde{Z}(d) Z\left(\frac{n}{d}\right)\quad (n \ge 2) \label{eqn:Ztilde_rec}
\end{align}
Then we can apply ``M\"obius inversion,'' that is, take the Dirichlet convolution of \(\widetilde{Z}(n)\) with \eqref{eqn:coeffsystem1} to get
\begin{equation}
    \begin{pmatrix}
        \pi(N; m, 1) \\ \vdots \\ \pi(N; m, a) \\ \vdots
    \end{pmatrix} = \frac{1}{N} \sum_{d \mid N} \widetilde{Z}(d) \begin{pmatrix}
        q^{\frac{N}{d}} - s_{m, \frac{N}{d}} \\ \vdots \\ c_{\frac{N}{d}}(\chi) \\ \vdots
    \end{pmatrix}
\end{equation}
and if we write the \((a, \chi)\)-th entry of \(\widetilde{Z}(d)\) as \(\widetilde{Z}(d)_{a, \chi}\), then this gives:
\begin{theorem}
    \label{thm:picformula}
    Let $N \ge 1$, $m \in \bF_q[T]$, and $a \in (\bF_q[T] / m \bF_q[T])^\times$ be a congruence class.
    Then
    \begin{equation}
    \label{eqn:picformula1}
        \pi(N; m, a) = \frac{1}{N} \sum_{d \mid N} \left(\widetilde{Z}(d)_{a, \chi_0} (q^{\frac{N}{d}} - s_{m, \frac{N}{d}}) - \sum_{\chi \ne \chi_0} \widetilde{Z}(d)_{a, \chi}\sum_{j=1}^{d(\chi)} \alpha_j(\chi)^{\frac{N}{d}} \right).
    \end{equation}
\end{theorem}
Now, the main result of this section is a formula for \(\widetilde{Z}(n)\) (Proposition \ref{thm:Ztilde}).
\begin{theorem}
    \label{thm:Ztilde}
    The $(a, \chi)$-th entry of $\widetilde{Z}(n)$ is given by
    \begin{equation}
    \label{eqn:Ztilde}
        \widetilde{Z}(n)_{a, \chi} = \frac{\mu(n)}{M'} \sum_{\substack{b \in \scU_m \\ b^n = a}} \chi(b)^{-1}.
    \end{equation}
\end{theorem}
\begin{proof}
The $n = 1$ case follows from \eqref{eqn:Ztilde_one}.
For $n \ge 2$, since $\widetilde{Z}(n)$ is uniquely determined by the recurrence relation \eqref{eqn:Ztilde_rec}, it suffices to show that
\begin{align*}
    \left(\sum_{d \mid n} \widetilde{Z}(d) Z\left(\frac{n}{d}\right)\right)_{(c_1, c_2)} = \sum_{\chi \in \what{\scU}_m}\sum_{d \mid n} \widetilde{Z}(d)_{c_1, \chi} Z\left(\frac{n}{d}\right)_{\chi, c_2} = 0
\end{align*}
for all $c_1, c_2 \in \scU_m$, assuming \eqref{eqn:Ztilde}.
Expanding this gives
\begin{equation}
\label{eqn:Ztilde_eq1}
    \sum_{\chi \in \what{\scU}_m} \sum_{d \mid n} \frac{\mu(d)}{M'} \sum_{\substack{b \in \scU_m \\ b^d = c_1}} \chi(b)^{-1} \chi(c_2)^{n/d} = \frac{1}{M'} \sum_{d \mid n} \mu(d) \sum_{\substack{b \in \scU_m \\ b^d = c_1}} \sum_{\chi \in \what{\scU}_m} \chi(b^{-1} c_2^{n/d})
\end{equation}
By the orthogonality of characters, the innermost sum on the right-hand side of \eqref{eqn:Ztilde_eq1} is
\[
\sum_{\chi \in \what{\scU}_m} \chi(b^{-1} c_2^{n/d}) = \begin{cases}
    M' & b = c_2^{n/d} \\ 0 & b \ne c_2^{n/d}
\end{cases}
\]
and plugging this in \eqref{eqn:Ztilde_eq1} gives
\begin{align*}
    \sum_{\chi \in \what{\scU}_m} \sum_{d \mid n} \frac{\mu(d)}{M'} \sum_{\substack{b \in \scU_m \\ b^d = c_1}} \chi(b)^{-1} \chi(c_2)^{n/d} &= \frac{1}{M'} \sum_{d \mid n} \mu(d) M' \delta_{c_1, c_2^n}  = \delta_{c_1, c_2^n} \sum_{d \mid n} \mu(d) = 0.
\end{align*}
\end{proof}

\begin{corollary}
    \label{cor:Ztilde}
    \begin{enumerate}
        \item If \(\gcd(n, M') = 1\), then
        \begin{equation}
        \label{eqn:Ztildecoprime}
            \widetilde{Z}(n) = \frac{\mu(n)}{M'} Z(n^{-1})^{H}, \quad (\widetilde{Z}(n))_{a, \chi} = \frac{\mu(n)}{M'} \chi(a)^{-n^{-1}}
        \end{equation}
        where \(n^{-1}\) is the inverse of \(n\) modulo \(M'\).
        \item For a prime \(\ell\) dividing \(M'\), we have
        \begin{equation}
        \label{eqn:Ztildeprimediv}
            \widetilde{Z}(\ell)_{a, \chi} = -\frac{\ell}{M'} \times \begin{cases} \chi(a^{1/\ell})^{-1} & a \in \scU_m^\ell, \, \chi \in \what{\scU}_m^\ell \\ 0 & \text{otherwise} \end{cases}
        \end{equation}
        where $\scU_m^\ell \subset \scU_m$ (resp. $\what{\scU}_m^\ell \subset \what{\scU}_m$) is image of the $\ell$-power map.
        Note that $\chi(a^{1/\ell})$ does not depend on the choice of $\ell$-th root $a^{1/\ell}$ of $a$.
        \item Assume $\scU_m$ is cyclic. Let $g = \gcd(n, M')$. We have
        \begin{equation}
        \label{eqn:Ztildecyclic}
            \widetilde{Z}(n)_{a, \chi} = \frac{\mu(n)g}{M'} \times \begin{cases} 
            \chi(a^{1/g})^{-(n/g)^{-1}} & a \in \scU_m^g, \, \chi \in \what{\scU}_m^g \\ 0 & \text{otherwise}
            \end{cases}
        \end{equation}
        where $\scU_m^g \subset \scU_m$ (resp. $\what{\scU}_m^g \subset \what{\scU}_m$) is image of the $g$-power map, and $(n/g)^{-1}$ is the inverse of $n/g$ modulo $M'$.
        Note that $\widetilde{Z}(n)_{a, \chi}$ does not depend on the choice of the $g$-th root $a^{1/g}$ of $a$.
    \end{enumerate}
\end{corollary}
\begin{proof}
\begin{enumerate}
    \item When $n$ and $M'$ are coprime, the $n$-power map $x \mapsto x^n$ is an isomorphism from $\scU_m$ to itself.
    Thus for a given $a \in \scU_m$ and $\chi \in \what{\scU}_m$, there exists a unique $b \in \scU_m$ with $b^n = a$, and $a^{n^{-1}} = (b^n)^{n^{-1}} = b^{n \cdot n^{-1}} = b$.
    Then \eqref{eqn:Ztilde} gives
    \[
        \widetilde{Z}(n)_{a, \chi} = \frac{\mu(n)}{M'} \chi(b)^{-1} = \frac{\mu(n)}{M'} \chi(b^{-1}) = \frac{\mu(n)}{M'} \chi(a^{-n^{-1}}) = \frac{\mu(n)}{M'} \chi(a)^{-n^{-1}}.
    \]
    \item When a prime $\ell$ divides $M'$, $\scU_m^\ell \subset \scU_m$ is an index-$\ell$ subgroup of $\scU_m$, and for each $a \in \scU_m^\ell$, there are $\ell$ elements $b \in \scU_m$ with $b^\ell = a$.
    If $b_1^\ell = b_2^\ell = a$ and $\chi = \chi_1^\ell \in \what{\scU}_m^\ell$, $(b_1^{-1} b_2)^{\ell} = 1$ and $\chi(b_1^{-1} b_2) = \chi_1(b_1^{-1}b_2)^\ell = 1$, so $\chi(b_1) = \chi(b_2)$.
    Hence $\chi(a^{1/\ell})$ does not depend on the choice of $a^{1/\ell}$ and $\mu(\ell) = -1$ with \eqref{eqn:Ztilde} gives \eqref{eqn:Ztildeprimediv}.
    \item The proof is similar.
    First, if $a = b^n$ for some $b \in \scU_m$, then $a = (b^{n/g})^g \in \scU_m^g$ and $\widetilde{Z}(n)_{a, \chi} = 0$ for $a \not \in \scU_m^g$.
    Since $\scU_m$ is cyclic, the kernel of the $g$-power map has exactly $g$ elements, and there are $g$ elements $c \in \scU_m$ with $c^g = a$.
    For each $c$, since $\gcd(n/g, M') = 1$, there exists a unique $b$ with $b^{n/g} = c$ (namely $b = c^{(n/g)^{-1}}$), so there are $g$ terms in the sum \eqref{eqn:Ztilde}, and \eqref{eqn:Ztildecyclic} follows.
\end{enumerate}
\end{proof}

The formula \eqref{eqn:picformula1} can be simplified further when $\scU_m$ (and hence $\what{\scU}_m$) is cyclic.

\begin{theorem}
    \label{thm:picformula_cyclic}
    Let $N \ge 1$, let $m \in \bF_q[T]$, and assume that $\scU_m = (\bF_q[T] / m \bF_q[T])^\times$ is a cyclic group generated by $c$, with $M' = \Phi(m) = |\scU_m|$.
    Then for $k \in \bZ$, we have
    \begin{equation}
        \label{eqn:piformula_cyclic}
        \pi(N; m, c^k) = \frac{1}{M'N} \sum_{g \mid M'} g \delta_{g \mid k} \sum_{\substack{d \mid N \\ \gcd(d, M') = g}} \mu(d) \left(q^{\frac{N}{d}} - s_{m, \frac{N}{d}} + \sum_{j=1}^{\frac{M'}{g} - 1} \zeta_{M'}^{-kj\left(\frac{d}{g}\right)^{-1}} c_{\frac{N}{d}}(\chi_1^{gj})\right)
    \end{equation}
\end{theorem}

\begin{proof}
Since $\scU_m$ is cyclic, the character group $\what{\scU}_m$ is also cyclic, and we denote by $\chi_1$ a generator of $\what{\scU}_m$.
We can decompose $\pi(N; m, a)$ as
\begin{equation}
\label{eqn:pi_decomp_pig}
    \pi(N; m, a) = \sum_{g \mid M'} \pi_g(N; m, a)
\end{equation}
where
\begin{align}
\label{eqn:pig_def}
    \pi_g(N; m, a) = \frac{1}{N} \sum_{\substack{d \mid N \\ \gcd(d, M') = g}} \left( \widetilde{Z}(d)_{a, \chi_0}(q^{\frac{N}{d}} - s_{m, \frac{N}{d}}) + \sum_{\chi \ne \chi_0} \widetilde{Z}(d)_{a, \chi} c_{\frac{N}{d}}(\chi)\right).
\end{align}
If $a = c^k$ with $0 \le k < M'$, the sum simply vanishes unless $g \mid k$.
For the latter case, only powers of $\chi_1^g$ contribute to the inner sum over nontrivial characters.
Then Corollary \ref{cor:Ztilde} gives
\begin{align}
    \pi_g(N; m, c^k) &= \frac{g \delta_{g \mid k}}{M'N} \sum_{\substack{d \mid N \\ \gcd(d, M') = g}} \mu(d) \left(q^{\frac{N}{d}} - s_{m, \frac{N}{d}} + \sum_{j=1}^{\frac{M'}{g} - 1} \chi_1^{gj}(c^{\frac{k}{g}})^{-\left(\frac{d}{g}\right)^{-1}} c_{\frac{N}{d}}(\chi_1^{gj})\right) \nonumber \\
    &= \frac{g \delta_{g \mid k}}{M'N} \sum_{\substack{d \mid N \\ \gcd(d, M') = g}} \mu(d) \left(q^{\frac{N}{d}} - s_{m, \frac{N}{d}} + \sum_{j=1}^{\frac{M'}{g} - 1} \zeta_{M'}^{-kj\left(\frac{d}{g}\right)^{-1}} c_{\frac{N}{d}}(\chi_1^{gj})\right) \label{eqn:pig_cyclic}
\end{align}
and summing this over all $g \mid M'$ proves \eqref{eqn:piformula_cyclic}.
\end{proof}

If $\gcd(N, M') = 1$, then $\gcd(d, M') = 1$ for all $d \mid N$ and
\begin{equation}
    \label{eqn:pi_coprime_deg}
    \pi(N; m, c^k) = \pi_1(N; m, c^k) = \frac{1}{M'N} \sum_{d \mid N} \mu(d)\left(q^{\frac{N}{d}} - s_{m, \frac{N}{d}} + \sum_{j=1}^{M' - 1} \zeta_{M'}^{-kjd^{-1}} c_{\frac{N}{d}}(\chi_1^{j})\right).
\end{equation}

Also, we record the following lemma, which will be used later in examples.
\begin{lemma}
\label{lem:mobius_sum_pnmid}
    Let $p$ be a prime. If $N \ne 1$ and $N \ne p$, then $\sum_{p \nmid d \mid N} \mu(d) = 0$.
\end{lemma}
\begin{proof}
    This follows from
    \begin{align*}
        \sum_{p \nmid d \mid N} \mu(d) &= \sum_{d \mid N} \mu(d) - \sum_{p \mid d \mid N} \mu(d) = \delta_{N = 1} - \delta_{p \mid N}\sum_{d \mid \frac{N}{p}} \mu(pd) = \delta_{N = 1} + \delta_{p \mid N} \delta_{N = p}.
    \end{align*}
\end{proof}

\subsection{Examples}

In the following examples, we show how the explicit formulas \eqref{eqn:pig_def} and \eqref{eqn:pig_cyclic} can be used to prove ties.

\subsubsection{\texorpdfstring{\(m = T^2 + T + 1 \in \bF_2[T]\)}{m = T^2 + T + 1 over characteristic 2}}
\label{subsubsec:p2T2T1_explicit}

Let \(\chi_1: (\bF_2[T] / m)^\times \to \bC^\times\) be the Dirichlet character with \(\chi_1(T) = \zeta_3\).
Then the \(L\)-functions of $\chi_1$ and $\chi_1^2$ are $\cL(u, \chi_1) = \cL(u, \chi_1^2) = 1 - u$ where $u = 1$ is the only zero and $c_n(\chi_1) = c_n(\chi_1^2) = -1$ for all $n$.

\paragraph{\textbf{When $N$ is not a multiple of 3:}}
When $N \equiv e \pmod{3}$ is not a multiple of 3, \eqref{eqn:pi_coprime_deg} gives 
\begin{align*}
\pi(N; m, T^k) &= \frac{1}{3N} \sum_{d \mid N} \mu(d) \left(2^{\frac{N}{d}} - 2\delta_{2 \mid \frac{N}{d}} - (\zeta_3^{-kd^{-1}} + \zeta_3^{-2kd^{-1}})\right) \\
&= \frac{1}{3N} \sum_{d \mid N} \mu(d) \left(2^{\frac{N}{d}} - 2\delta_{2 \mid \frac{N}{d}} - \underbrace{(\zeta_3^{-ke^{-1}\frac{N}{d}} + \zeta_3^{-2ke^{-1}\frac{N}{d}})}_{(*)}\right) \\
\end{align*}
where $e^{-1} \cdot e \equiv 1 \pmod{3}$. The quantity $(*)$ depends only on $k$ (and not on $e$, $N$, or $d$):
\[
(*) = \begin{cases}
    2 & k \equiv 0 \pmod{3} \\ -1 & k \not\equiv 0 \pmod{3}
\end{cases}
\]
Since $\sum_{d \mid N} \mu(d) = 0$ for $N > 1$, the sum of $\mu(d) \cdot (*)$ over $d \mid N$ vanishes, and this proves that $\pi(N; m, T^k)$ is the same for $k = 0, 1, 2$.

\paragraph{\textbf{When $N$ is a multiple of 3:}}
Assume that $N$ is a multiple of $3$.
Let $a = T^k \ne 1$, equivalently $k \not \equiv 0 \pmod{3}$.
Then $\pi_3(N; m, T^k) = 0$, and we get
\begin{equation}
\label{eqn:T2T1N3}
\pi(N; m, T^k) = \frac{1}{3N} \sum_{3 \nmid d \mid N} \mu(d) \left(2^{\frac{N}{d}} - 2\delta_{2 \mid \frac{N}{d}} - (\zeta_3^{-kd^{-1}} + \zeta_3^{-2kd^{-1}}) \right) = \frac{1}{3N} \sum_{3 \nmid d \mid N} \mu(d) (2^{\frac{N}{d}} - 2\delta_{2 \mid \frac{N}{d}}).
\end{equation}
Since the result is independent of $k$, we get $\pi(N; m, T) = \pi(N; m, T^2) = \pi(N; m, T + 1)$.

Table \ref{tab:T2T1group} summarizes the ties for each $N$ modulo 3.

\begin{remark}
\label{rmk:T2T1bias}
    For $N \equiv 0 \pmod{3}$, we also have a consistent \emph{bias} against $a = 1$.
    From \eqref{eqn:T2T1N3} and \eqref{eqn:gauss_cnt}, one can show that
    \begin{equation}
        \label{eqn:piT1diff}
        \pi(N; m, T) - \pi(N; m, 1) = - \frac{1}{N} \sum_{3 \mid d \mid N} \mu(d) (2^{\frac{N}{d}} - (-1)^{\frac{N}{d}}) - \frac{1}{N} \sum_{d \mid N} \mu(d) (-1)^{\frac{N}{d}}.
    \end{equation}
    Now, we have the following lemma:
    \begin{lemma}
    For $N \ge 1$,
    \begin{equation}
        \label{eqn:mobiussumalt}
        \sum_{d \mid N} \mu(d) (-1)^{\frac{N}{d}} = \begin{cases}
            -1 & N = 1\\ 2 & N = 2 \\ 0 & N \ge 3
        \end{cases}
    \end{equation}
    \end{lemma}
    \begin{proof}
    This immediately follows from
    \begin{align*}
        \sum_{n \ge 1} \frac{\mu(n)}{n^s} &= \frac{1}{\zeta(s)},\\
        \sum_{n \ge 1} \frac{(-1)^n}{n^s} &= -\frac{1}{1^s} + \frac{1}{2^s} - \frac{1}{3^s} + \frac{1}{4^s} - \cdots = 2\left(\frac{1}{2^s} + \frac{1}{4^s} + \cdots\right) - \zeta(s) = \left(-1 + \frac{2}{2^s}\right)\zeta(s)
    \end{align*}
    where \eqref{eqn:mobiussumalt} is the Dirichlet convolution of $\mu(n)$ and $(-1)^n$.
    \end{proof}
    Using \eqref{eqn:mobiussumalt}, we can simplify the difference \eqref{eqn:piT1diff} as
    \begin{align*}
        -\frac{1}{N} \sum_{3 \mid d \mid N} \mu(d)(2^{\frac{N}{d}} - (-1)^{\frac{N}{d}}) = -\frac{1}{N} \sum_{3 \mid d \mid N} \mu(d) 2^{\frac{N}{d}} + \frac{1}{N} \sum_{d \mid \frac{N}{3}} \mu(3d) (-1)^{\frac{N}{3d}}
    \end{align*}
    When $N = 3$ or $N = 6$, the difference is $1$ and $0$.
    When $N \ge 9$, the second summation vanishes by \eqref{eqn:mobiussumalt}, and
    \begin{align*}
        -\frac{1}{N} \sum_{3 \mid d \mid N} \mu(d) 2^{\frac{N}{d}}
        &\ge \frac{1}{N} (2^{\frac{N}{3}} - 2^{\frac{N}{3}-1} - 2^{\frac{N}{3}-2} - \cdots - 1) = \frac{1}{N} > 0.
    \end{align*}
    Thus we have $\pi(N; m, 1) \le \pi(N; m, T) = \pi(N; m, T + 1)$ for all $3 \mid N$, where the first inequality becomes equality only if $N = 6$.
\end{remark}

\subsubsection{\texorpdfstring{\(m = T^3 + T + 1 \in \bF_2[T]\)}{m = T^3 + T + 1 over characteristic 2}}
\label{subsubsec:p2T3T1}

Let \(\chi_1: (\bF_2[T] / m)^\times \to \bC^\times\) be the Dirichlet character with \(\chi_1(T) = \zeta_7 = e^{2\pi i / 7}\).
The corresponding \(L\)-function is
\[
\cL(u, \chi) = (1 - u)(1 - \alpha u), \quad \alpha = \zeta_7^2 + \zeta_7^4 + \zeta_7^5 + \zeta_7^6 = -1 - \zeta_7 - \zeta_7^3.
\]
We have \(|\alpha| = \sqrt{2}\), and $\alpha / \sqrt{2}$ is not a root of unity, which can be checked by computing its minimal polynomial.

\paragraph{\textbf{When $N$ is not a multiple of 7:}}
When $N \equiv e \pmod{7}$ is not a multiple of 7, \eqref{eqn:pi_coprime_deg} gives
\begin{align*}
    \pi(N; m, T^k) &= \frac{1}{7N} \sum_{d \mid N} \mu(d) \left(2^{\frac{N}{d}} - 3 \delta_{3 \mid \frac{N}{d}} - \sum_{l=1}^{6} \zeta_7^{-ke^{-1}l \frac{N}{d}} (1 + \sigma_l(\alpha)^{\frac{N}{d}})\right) \\
    &= \frac{1}{7N} \sum_{d \mid N} \mu(d) \left(2^{\frac{N}{d}} - 3 \delta_{3 \mid \frac{N}{d}} - \underbrace{\sum_{l=1}^{6} \zeta_7^{-ke^{-1}l \frac{N}{d}}}_{(*)}  - \Tr((\zeta_7^{-ke^{-1}} \alpha)^{\frac{N}{d}})\right)
\end{align*}
where $e^{-1} \cdot e \equiv 1 \pmod{7}$, $\Tr = \Tr_{\bQ(\zeta_7) / \bQ}$, and $(*)$ depends only on $k$:
\[
(*) = \begin{cases}
    6 & k \equiv 0 \pmod{7} \\ -1 & k \not\equiv 0 \pmod{7}
\end{cases}
\]
As in the previous example, the contribution of $(*)$ vanishes when $N > 1$.
Thus $\pi(N; m, T^k) = \pi(N; m, T^{k'})$ if
\[
\Tr((\zeta^{-ke^{-1}} \alpha)^{n}) = \Tr((\zeta^{-k'e^{-1}}\alpha)^{n})
\]
for all $n \mid N$.
Now, one can directly check that
\begin{equation}
    \label{eqn:T3T1conj}
\sigma_2(\alpha) = \zeta_7^4 + \zeta_7^8 + \zeta_7^{10} + \zeta_7^{12} = \zeta_{7} + \zeta_{7}^{3} + \zeta_{7}^{4} + \zeta_{7}^{5} = \zeta_{7}^{-1} \alpha,
\end{equation}
hence \(\alpha\) and \(\zeta_7^{-1} \alpha\) are conjugate to each other.
Similarly, \(\sigma_4(\alpha) = \zeta_7^{-3}\alpha\) shows that \(\zeta_7^{-3}\alpha\) is also conjugate to these two.
Thus traces of their powers are the same and
\[
\Tr(\alpha^\frac{N}{d}) = \Tr((\zeta_7^{-e \cdot e^{-1}} \alpha)^{\frac{N}{d}}) = \Tr((\zeta_7^{-3e \cdot e^{-1}} \alpha)^{\frac{N}{d}}),
\]
hence we get the ties
\[
\pi(N; m, 1) = \pi(N; m, T^e) = \pi(N; m, T^{3e}).
\]
Similarly, $\zeta_7^{-2} \alpha$, $\zeta_7^{-4} \alpha$, and $\zeta_{7}^{-5}\alpha$ are conjugate to each other, and this gives a second family of ties
\[
\pi(N; m, T^{2e}) = \pi(N; m, T^{4e}) = \pi(N; m, T^{5e}).
\]
When $e = 1$, the first group gives the tie \eqref{eqn:tieN1mod7} mentioned earlier.

\paragraph{\textbf{When $N$ is a multiple of 7:}}

Write $N = 7^r \cdot N_1$ for $r \ge 1$ and $7 \nmid N_1$.
Assume $7 \nmid k$.
Then $\pi_7(N; m, T^k) = 0$ and
\[
\pi(N; m, T^k) = \pi_1(N; m, T^k) = \frac{1}{7N} \sum_{7 \nmid d \mid N} \mu(d) \left(2^{\frac{N}{d}} - 3 \delta_{3 \mid \frac{N}{d}} - 6 - \Tr((\zeta_7^{-k N_1^{-1}} \alpha^{7^r})^{\frac{N_1}{d}})\right)
\]
where $N_1^{-1} \cdot N_1 \equiv 1 \pmod{7}$. From \eqref{eqn:T3T1conj},
\[
\sigma_2(\alpha^{7^r}) = \sigma_2(\alpha)^{7^r} = (\zeta_7^{-1} \alpha)^{7^r} = \alpha^{7^r},
\]
i.e. $\alpha^{7^r}$ is fixed under $\sigma_2$\footnote{We have $\alpha^7 = -\frac{13}{2} - \frac{7\sqrt{7}}{2}i$.}.
Hence for $t \not\equiv 0 \pmod{7}$, $\zeta_7^t \alpha^{7^r}$ and $\sigma_2(\zeta^{t} \alpha^{7^r}) = \zeta^{2t} \alpha^{7^r}$ are conjugate and we get the ties
\begin{align*}
    \pi(N; m, T) = \pi(N; m, T^2) = \pi(N; m, T^4), \quad \pi(N; m, T^3) = \pi(N; m, T^6) = \pi(N; m, T^5).
\end{align*}
These ties coincide with the patterns shown in Table \ref{tab:T3T1}.

\subsubsection{\texorpdfstring{\(m = T^2 + 1 \in \bF_3[T]\)}{m = T^2 + 1 over characteristic 3}}
\label{subsubsec:p3T21explicit}

Let $\chi_1 : (\bF_3[T] / m)^\times \to \bC^\times$ be the Dirichlet character with $\chi_1(T + 1) = \zeta_8 = e^{2\pi i / 8}$ (note that $T + 1$ is a generator of $(\bF_3[T] / m)^\times$, but $T$ is not).
For $1 \le l \le 7$, the $L$-functions of $\chi_1^l$ are
\begin{align*}
    \cL(u, \chi_1^l) = \begin{cases}
        1 - \sigma_l(\alpha)u & 2 \nmid l \\
        1 - u & 2 \mid l
    \end{cases}
\end{align*}
for $\alpha = \zeta_8^2 + \zeta_8^3 + \zeta_8^5 = -\sqrt{2} + i$, which has an irrational argument and $\alpha / \sqrt{3}$ is not a root of unity.

\paragraph{\textbf{When $N$ is odd:}}
For odd $N$, \eqref{eqn:pi_coprime_deg} gives
\begin{align*}
\pi(N; m, (T + 1)^k) &= \pi_1(N; m, (T+1)^k)\\
&= \frac{1}{8N} \sum_{d \mid N} \mu(d) \left(3^{\frac{N}{d}} - 2 \delta_{2 \mid \frac{N}{d}} - \sum_{l=1}^{7} \zeta^{-kld^{-1}} \alpha(\chi^l)^{\frac{N}{d}}\right) \\
&=\frac{1}{8N} \sum_{d \mid N} \mu(d) \left(3^{\frac{N}{d}} - \underbrace{\left( \zeta_8^{-2kd^{-1}} + \zeta_8^{-4kd^{-1}} + \zeta_8^{-6kd^{-1}}+ \sum_{i=0}^{3} \zeta_8^{-k(2i+1)d^{-1}} \sigma_{2i+1}(\alpha)^{\frac{N}{d}}\right)}_{(*)} \right)
\end{align*}
and we only need to focus on the terms in (*) to check for ties.
If $N \equiv e \pmod 8$, $(*)$ can be rewritten as
\[
\zeta_8^{-2kd^{-1}} + \zeta_8^{-4kd^{-1}} + \zeta_8^{-6kd^{-1}} + \Tr((\zeta_8^{-ke^{-1}} \alpha)^{\frac{N}{d}}).
\]
One can check that the sum of the first three terms is
\begin{equation}
    \label{eqn:p3T21_sum1}
    \zeta_8^{-2kd^{-1}} + \zeta_8^{-4kd^{-1}} + \zeta_8^{-6kd^{-1}} = \begin{cases} 3 & k \equiv 0 \pmod{4} \\ -1 & k \not\equiv 0 \pmod{4} \end{cases}
\end{equation}
We have the following conjugate relation:
\[
\sigma_3(\alpha) = \zeta_8^6 + \zeta_8^9 + \zeta_8^{15} = \zeta_8 + \zeta_8^6 + \zeta_8^7 = \zeta_8^{-4} \alpha = -\alpha
\]
which implies $\sigma_3(\zeta_8^{-k}\alpha) = \zeta_8^{-3k-4}\alpha = \zeta_8^{-7k}\alpha$ and $\sigma_3(\zeta_8^{-3k}\alpha) = \zeta_8^{-9k-4}\alpha = \zeta_8^{-5k}\alpha$ for odd $k$.
From these equations, we get the ties
\begin{align*}
    \pi(N; m, 1) =\pi(N; m, (T+1)^0) &= \pi (N; m, (T+1)^4) = \pi(N; m, 2), \\
    \pi(N; m, T + 1) = \pi(N; m, (T+1)^1) &= \pi(N; m, (T+1)^7) = \pi(N; m, T + 2),\\
    \pi(N; m, 2T+1) = \pi(N; m, (T+1)^3) &= \pi(N; m, (T+1)^5) = \pi(N; m, 2T+2)
\end{align*}
for any odd $N$.

\paragraph{\textbf{When $N$ is even:}}
Now assume $N$ is even.
Write $N = 2^r \cdot N_1$ for $r \ge 1$ and $2 \nmid N_1$.
Assume $a = (T + 1)^k$ for odd $k$.
Then $\pi(N; m, (T+1)^k) = \pi_1(N; m, (T+1)^k)$ and only odd divisors of $N$ contribute to $\pi(N; m, (T+1)^k)$.
For odd $d$, the terms (*) become
\[
\zeta_8^{-2kd^{-1}} + \zeta_8^{-4kd^{-1}} + \zeta_8^{-6kd^{-1}} + \Tr((\zeta_8^{-kN_1^{-1}}\alpha^{2^r})^{\frac{N_1}{d}})
\]
where $N_1^{-1} \cdot N_1 \equiv 1 \pmod{8}$.
Since $r \ge 1$,
\[
\sigma_3(\alpha^{2^r}) = \sigma_3(\alpha)^{2^r} = (-\alpha)^{2^r} = \alpha^{2^r}
\]
and $\sigma_3(\zeta_8^{-k N_1}\alpha^{2^r}) = \zeta_8^{-3kN_1}\alpha^{2^r}$.
Combined with \eqref{eqn:p3T21_sum1}, this shows ties between congruence classes $(T+1)^{k}$ and $(T+1)^{3k}$, i.e.
\begin{align*}
    \pi(N; m, T + 1) = \pi(N; m, (T+1)^1) &= \pi(N; m, (T+1)^3) = \pi(N; m, 2T+1) \\
    \pi(N; m, 2T + 2) = \pi(N; m, (T+1)^5) &= \pi(N; m, (T+1)^7) = \pi(N; m, T+2).
\end{align*}

When $k$ is even, the contributions from odd divisors (that is, $\pi_1(N; m, (T+1)^k)$) stay the same, and they give the same value for $k$ and $3k$.
We have
\begin{align*}
    \pi(N; m, (T+1)^k) &= \pi_1(N; m, (T+1)^k) + \pi_2(N; m, (T+1)^k) \\
    &=\frac{1}{8N} \sum_{2 \nmid d \mid N} \mu(d) \left(3^{\frac{N}{d}} - 2 - \left( \zeta_8^{-2kd^{-1}} + \zeta_8^{-4kd^{-1}} + \zeta_8^{-6kd^{-1}}+ \sum_{i=0}^{3} \zeta_8^{-k(2i+1)d^{-1}} \sigma_{2i+1}(\alpha)^{\frac{N}{d}}\right)\right) \\
    &\quad + \frac{1}{8N} \sum_{2 \mid d \mid N} \mu(d) \left(3^{\frac{N}{d}} - (\zeta_8^{-2k(d/2)^{-1}} + \zeta_8^{-4k(d/2)^{-1}} + \zeta_8^{-6k(d/2)^{-1}})\right)
\end{align*}
which gives the same value for $k$ and $3k$.
This proves the tie
\[
\pi(N; m, 2T) = \pi(N; m, (T+1)^2) = \pi(N; m, (T+1)^6) = \pi(N; m, T).
\]

Table \ref{tab:p3T21group} shows the above ties for each $N$ modulo 2.

\subsubsection{\texorpdfstring{\(m = T^2 \in \bF_p[T]\)}{m = T^2 over characteristic p}}

When $p = 2$, we have $M' = 2$ and $T + 1$ is a generator of $\scU_m$, where $\chi_1 \in \what{\scU}_m$ with $\chi_1(T+1) = -1$ becomes a generator of $\what{\scU}_m$.
Also, $s_{T^2, n} = 1$ for any $n$, since the only irreducible factor $T \mid m$ has degree 1.
The corresponding $L$-function is $\cL(u, \chi_1) = 1 - u$ and $c_n(\chi_1) = -1$ for all $n$.
From \eqref{eqn:pi_coprime_deg}, for odd $N \ge 3$ we have a tie
\[
    \pi(N; m, 1) = \frac{1}{2N} \sum_{d \mid N} \mu(d) \left(2^{\frac{N}{d}} - 1 +(-1)\right) = \frac{1}{2N} \sum_{d \mid N} \left(2^{\frac{N}{d}} - 1 - (-1)\right) = \pi(N; m, T + 1)
\]
since $\sum_{d \mid N} \mu(d) = 0$.
Table \ref{tab:p2T2} shows the number of primes over $\bF_2$ that are $1$ or $T + 1$ modulo $T^2$.

\begin{remark}
\label{rmk:T2p2bias}
    As in Remark \ref{rmk:T2T1bias}, we have a consistent bias for $p = 2$ toward $T + 1$ when $N$ is even and $N > 2$.
    We have
    \begin{align*}
        \pi(N; m, 1) &= \pi_1(N; m, 1) + \pi_2(N; m, 1) = \frac{1}{2N} \sum_{2 \nmid d \mid N} \mu(d) \left(2^{\frac{N}{d}} - 1 - (-1)^{\frac{N}{d}}\right) + \frac{1}{N} \sum_{2 \mid d \mid N} \mu(d) \left(2^{\frac{N}{d}} - 1\right) \\
        \pi(N; m, T + 1) &= \pi_1(N; m, T + 1) = \frac{1}{2N} \sum_{2 \nmid d \mid N} \mu(d) \left(2^{\frac{N}{d}} - 1 + (-1)^{\frac{N}{d}}\right)
    \end{align*}
    By Lemma \ref{lem:mobius_sum_pnmid}, $\pi_1(N; m, 1) = \pi_1(N; m, T + 1)$ for $N > 2$.
    Now,
    \[
    \pi_2(N; m, 1) = \frac{1}{N} \sum_{d \mid \frac{N}{2}} \mu(2d) \left(2^{\frac{N}{2d}} - 1\right) = - \frac{1}{N} \sum_{d \mid \frac{N}{2}} \mu(d) 2^{\frac{N}{2d}} \le -\frac{1}{N} \left(2^{\frac{N}{2}} - 2^{\frac{N}{2} - 1} - 2^{\frac{N}{2} - 2} - \cdots - 1\right) = -\frac{1}{N} < 0,
    \]
    and we get $\pi(N; m, 1) < \pi(N; m, T + 1)$.
\end{remark}

When $p = 3$, we have $M' = 6$ and $c = T + 2$ is a generator of $\scU_m$, where $\chi_1 \in \what{\scU}_m$ with $\chi_1(T + 2) = \zeta_6$ becomes a generator of $\what{\scU}_m$. We also have $s_{m, n} = 1$ for all $n$.
The corresponding $L$-functions are
\begin{equation}
\label{eqn:p2T2Lftn}
    \cL(u, \chi^l) = \begin{cases}
        1 + (2\zeta_6 - 1)u = 1 + \sqrt{-3} u & l = 1 \\
        1 - u & l = 2 \\
        1 & l = 3 \\
        1 - u & l = 4 \\
        1 - (2\zeta_6 - 1)u = 1 - \sqrt{-3} u & l = 5
    \end{cases}
\end{equation}

\paragraph{\textbf{When $N$ is odd:}}

If $N$ is odd and $N > 3$, \eqref{eqn:pi_decomp_pig} gives
\[
\pi(N; m, (T+2)^k) = \pi_{1}(N; m, (T+2)^k) + \pi_3(N; m, (T+2)^k)
\]
where $\pi_1$ and $\pi_3$ are the summands defined in \eqref{eqn:pig_def}. By \eqref{eqn:pig_cyclic}, we have
\begin{align*}
    \pi_1(N; m, (T + 2)^k) 
    &= \frac{1}{6N} \sum_{3 \nmid d \mid N} \mu(d) \left(3^{\frac{N}{d}} - 1 - \left(\zeta_6^{-k d^{-1}}(-\sqrt{-3})^{\frac{N}{d}} + \zeta_6^{-2kd^{-1}} + \zeta_6^{-4kd^{-1}} + \zeta_6^{-5kd^{-1}} (\sqrt{-3})^{\frac{N}{d}}\right)\right) \\
    \pi_3(N; m, (T + 2)^k) &= \frac{\delta_{3 \mid k}}{2N} \sum_{3 \mid d \mid N} \mu(d) \left(3^{\frac{N}{d}} - 1\right) = \frac{\delta_{3 \mid k}}{2N} \sum_{d \mid \frac{N}{3}} \mu(3d) \left(3^{\frac{N}{3d}} - 1\right) = - \frac{\delta_{3 \mid k}}{2N}\sum_{d \mid \frac{N}{3}} \mu(d) 3^{\frac{N}{3d}}.
\end{align*}
In $\pi_3$, the summation over $\chi \ne \chi_0$ vanishes, since the $L$-function of $\chi = \chi_1^3$ has no zeros.
In particular, $\pi_3(N; m, (T+2)^k) = 0$ for $3 \nmid k$, and
\begin{equation}
\label{eqn:pi3_3divk}    
\pi_3(N; m, 1) = \pi_3(N; m, (T+2)^3).
\end{equation}

Assume $3 \nmid k$. Then $\pi_2(N; m, (T+2)^k) = 0$ and $1 + \zeta_6^{-2kd^{-1}} + \zeta_6^{-4kd^{-1}} = 0$, so
\begin{align*}
    \pi_1(N; m, (T+2)^k) = \frac{1}{6N} \sum_{3 \nmid d \mid N} \mu(d) \left(3^{\frac{N}{d}} - \zeta_6^{-kd^{-1}}(-\sqrt{-3})^{\frac{N}{d}} - \zeta_6^{-5kd^{-1}} (\sqrt{-3})^{\frac{N}{d}}\right)
\end{align*}
Now we have a conjugate relation for the zeros $\pm \sqrt{-3}$, namely
\[
\sigma_5(\sqrt{-3}) = \sigma_5(-1 + 2 \zeta_6) = -1 + 2 \zeta_6^5 = -\sqrt{-3} = \zeta_6^3 \sqrt{-3}
\]
and $\sigma_5(-\sqrt{-3}) = \sqrt{-3}$, where $\sigma_5: \zeta_6 \mapsto \zeta_6^5$ is simply the complex conjugation.
Thus
\begin{align*}
    \pi_1(N; m, (T+2)^{5k + 3}) &= \frac{1}{6N} \sum_{3 \nmid d \mid N} \mu(d) \left(3^{\frac{N}{d}} - \zeta_6^{-(5k+3)d^{-1}}(-\sqrt{-3})^{\frac{N}{d}} - \zeta_6^{-5(5k+3)d^{-1}} (\sqrt{-3})^{\frac{N}{d}}\right) \\
    &= \frac{1}{6N} \sum_{3 \nmid d \mid N} \mu(d) \left(3^{\frac{N}{d}} + \zeta_6^{-5kd^{-1}}(-\sqrt{-3})^{\frac{N}{d}} + \zeta_6^{-25kd^{-1}} (\sqrt{-3})^{\frac{N}{d}}\right)\\
    &= \frac{1}{6N} \sum_{3 \nmid d \mid N} \mu(d) \left(3^{\frac{N}{d}} - \zeta_6^{-kd^{-1}}(-\sqrt{-3})^{\frac{N}{d}} - \zeta_6^{-5kd^{-1}} (\sqrt{-3})^{\frac{N}{d}}\right)\\
    &= \pi_1(N; m, (T+2)^k)
\end{align*}
(Note that $d$ and $\frac{N}{d}$ are odd, since we are assuming that $N$ is odd.)
For $k = 1$ and $k = 4$, this gives
\[
\pi(N; m, T+2) = \pi(N; m, (T+2)^2), \quad \pi(N; m, (T+2)^4) = \pi(N; m, (T+2)^5).
\]

For $3 \mid k$, we will use the following lemma.

Now, assume $3 \mid k$ and $N \ne 1, 3$. Then Lemma \ref{lem:mobius_sum_pnmid} with $p = 3$ gives
\begin{align*}
    \pi_1(N; m, 1) &= \frac{1}{6N} \sum_{3 \nmid d \mid N} \mu(d) \left(3^{\frac{N}{d}} - 1 - \left((-\sqrt{-3})^{\frac{N}{d}} + 2 + (\sqrt{-3})^{\frac{N}{d}}\right)\right) \\
    &= \frac{1}{6N} \sum_{3 \nmid d \mid N} \mu(d) 3^{\frac{N}{d}} \\
    &= \frac{1}{6N} \sum_{3 \nmid d \mid N} \mu(d) \left(3^{\frac{N}{d}} - 1 - \left(-(-\sqrt{-3})^{\frac{N}{d}} - 2 - (\sqrt{-3})^{\frac{N}{d}}\right)\right) = \pi_1(N; m, (T+2)^3)
\end{align*}
since $\frac{N}{d}$ is odd and $(\sqrt{-3})^{\frac{N}{d}} + (-\sqrt{-3})^{\frac{N}{d}} = 0$.
Combined with \eqref{eqn:pi3_3divk}, this proves $\pi(N; m, 1) = \pi(N; m, (T+2)^3)$.

\paragraph{\textbf{When $N$ is even:}} As above, \eqref{eqn:pi_decomp_pig} gives
\[
\pi(N; m, (T+2)^k) = \pi_1(N; m, (T+2)^k) + \pi_2(N; m, (T+2)^k) + \pi_3(N; m, (T+2)^k) + \pi_6(N; m, (T+2)^k)
\]
where
\begin{align*}
    \pi_g(N; m, (T+2)^k) = \frac{g \delta_{g \mid k}}{6N} \sum_{\substack{d \mid N \\ \gcd(d, 6) = g}} \mu(d) \left(3^{\frac{N}{d} }- 1 + \sum_{j=1}^{6/g - 1} \zeta_6^{-kj\left(\frac{d}{g}\right)^{-1}}c_{\frac{N}{d}}(\chi^{gj})\right)
\end{align*}
for $g = 1, 2, 3, 6$.
More precisely, we have
\begin{align*}
    \pi_1(N; m, (T+2)^k) &= \frac{1}{6N} \sum_{\substack{d \mid N \\ \gcd(d, 6) = 1}} \mu(d) \left(3^{\frac{N}{d}} - 1 + \sum_{j=1}^{5} \zeta_6^{kjd^{-1}} c_{\frac{N}{d}}(\chi_1^j)\right) \\
    &= \frac{1}{6N} \sum_{\substack{d \mid N \\ \gcd(d, 6) = 1}} \mu(d) \left(3^{\frac{N}{d}} - 1 - \left(\zeta_6^{-k d^{-1}}(-\sqrt{-3})^{\frac{N}{d}} + \zeta_6^{-2kd^{-1}} + \zeta_6^{-4kd^{-1}} + \zeta_6^{-5kd^{-1}} (\sqrt{-3})^{\frac{N}{d}}\right)\right) \\
    \pi_2(N; m, (T+2)^k) &= \frac{\delta_{2 \mid k}}{3N} \sum_{\substack{d \mid N \\ \gcd(d, 6) = 2}} \mu(d) \left(3^{\frac{N}{d}} - 1 + \sum_{j=1}^{2} \zeta_6^{-kj\left(\frac{d}{2}\right)^{-1}} c_{\frac{N}{d}}(\chi_1^{2j})\right) \\
    &= \frac{\delta_{2 \mid k}}{3N} \sum_{\substack{d \mid N \\ \gcd(d, 6) = 2}} \mu(d)\left(3^{\frac{N}{d}} - 1 - \zeta_6^{-k\left(\frac{d}{2}\right)^{-1}} - \zeta_6^{-2k\left(\frac{d}{2}\right)^{-1}}\right) \\
    \pi_3(N; m, (T+2)^k) &= \frac{\delta_{3 \mid k}}{2N} \sum_{\substack{d \mid N \\ \gcd(d, 6) = 3}} \mu(d) \left(3^{\frac{N}{d}} - 1\right) \\
    \pi_6(N; m, (T+2)^k) &= \frac{\delta_{6 \mid k}}{N} \sum_{6 \mid d \mid N} \mu(d) \left(3^{\frac{N}{d}} - 1\right)
\end{align*}
When $k = 2$ or $k = 4$, $1 + \zeta_6^{-k\left(\frac{d}{2}\right)^{-1}} + \zeta_6^{-2k\left(\frac{d}{2}\right)^{-1}} = 0$ for $\gcd(d, 6) = 2$ and $\pi_2$ further simplifies as
\begin{align*}
    \pi_2(N; m, (T+2)^k) = \frac{1}{3N} \sum_{\substack{d \mid N \\ \gcd(d, 6) = 2}} \mu(d) 3^{\frac{N}{d}} = \frac{1}{3N} \sum_{3 \nmid d \mid \frac{N}{2}} \mu(2d) 3^{\frac{N}{2d}} = -\frac{1}{3N} \sum_{3 \nmid d \mid \frac{N}{2}} \mu(d) 3^{\frac{N}{2d}}
\end{align*}

If $k = 1$ or $k = 5$, $\pi_2 = \pi_3 = \pi_6 = 0$ and 
\begin{align*}
    \pi(N; m, T+2) &= \pi_1(N; m, T+2) \\
    &=\frac{1}{6N} \sum_{\substack{d \mid N \\ \gcd(d, 6) = 1}} \mu(d) \left(3^{\frac{N}{d}} - 1 - (\zeta_6^{-d^{-1}} (-\sqrt{-3})^{\frac{N}{d}} + \zeta_6^{-2d^{-1}} + \zeta_6^{-4d^{-1}} + \zeta_6^{-5d^{-1}} (\sqrt{-3})^{\frac{N}{d}})\right) \\
    \pi(N; m, (T+2)^5) &= \pi_1(N; m, (T+2)^5) \\
    &=\frac{1}{6N} \sum_{\substack{d \mid N \\ \gcd(d, 6) = 1}} \mu(d) \left(3^{\frac{N}{d}} - 1 - (\zeta_6^{-5d^{-1}} (-\sqrt{-3})^{\frac{N}{d}} + \zeta_6^{-10d^{-1}} + \zeta_6^{-20d^{-1}} + \zeta_6^{-25d^{-1}} (\sqrt{-3})^{\frac{N}{d}})\right) \\
    &=\frac{1}{6N} \sum_{\substack{d \mid N \\ \gcd(d, 6) = 1}} \mu(d) \left(3^{\frac{N}{d}} - 1 - (\zeta_6^{-5d^{-1}} (-\sqrt{-3})^{\frac{N}{d}} + \zeta_6^{-4d^{-1}} + \zeta_6^{-2d^{-1}} + \zeta_6^{-d^{-1}} (\sqrt{-3})^{\frac{N}{d}})\right)
\end{align*}
Since all $d$'s appearing in the above sums are odd, $\frac{N}{d}$'s are all even and
\[
\zeta_6^{-d^{-1}} (-\sqrt{-3})^{\frac{N}{d}} + \zeta_6^{-5d^{-1}} (\sqrt{-3})^{\frac{N}{d}} = (\zeta_6^{-d^{-1}} + \zeta_6^{-5d^{-1}}) (\sqrt{-3})^{\frac{N}{d}} = \zeta_6^{-5d^{-1}} (-\sqrt{-3})^{\frac{N}{d}} + \zeta_6^{-d^{-1}} (\sqrt{-3})^{\frac{N}{d}},
\]
hence we get $\pi(N; m, T+2) = \pi(N; m, (T+2)^5)$.

If $k = 2$ and $k = 4$, we have $\pi_1(N; m, (T+2)^2) = \pi_1(N; m, (T+2)^4)$ and
\[
\zeta_6^{2\left(\frac{d}{2}\right)^{-1}} + \zeta_6^{4\left(\frac{d}{2}\right)^{-1}} = \zeta_6^{4\left(\frac{d}{2}\right)^{-1}} + \zeta_6^{8\left(\frac{d}{2}\right)^{-1}}
\]
implies $\pi_2(N; m, (T+2)^2) = \pi_2(N; m, (T+2)^4)$, hence $\pi(N; m, (T+2)^2) = \pi(N; m , (T+2)^4)$.

\paragraph{\textbf{When $N$ is a multiple of 4:}} Once we further assume that $N \equiv 0 \pmod{4}$, we can show that
\[
\pi(N; m, T+2) = \pi(N; m, (T+2)^2) = \pi(N; m, (T+2)^4) = \pi(N; m, (T+2)^5).
\]
Since we have already shown the ties for $k = 1, 5$ and $k = 2, 4$, it is enough to show that the first two are the same.
We have
\begin{align*}
    \pi(N; m, T+2) &= \sum_{\substack{d \mid N \\ \gcd(d, 6) = 1}} \mu(d) \left(3^{\frac{N}{d}} - 1 - (\zeta_6^{-d^{-1}}(-\sqrt{-3})^{\frac{N}{d}} + \zeta_6^{-2d^{-1}} + \zeta_6^{-4d^{-1}} + \zeta_6^{-5d^{-1}}(\sqrt{-3})^{\frac{N}{d}}) \right) \\
    \pi(N; m, (T+2)^2) &= \sum_{\substack{d \mid N \\ \gcd(d, 6) = 1}} \mu(d) \left(3^{\frac{N}{d}} - 1 - (\zeta_6^{-2d^{-1}}(-\sqrt{-3})^{\frac{N}{d}} + \zeta_6^{-4d^{-1}} + \zeta_6^{-8d^{-1}} + \zeta_6^{-10d^{-1}}(\sqrt{-3})^{\frac{N}{d}}) \right) \\
    &\quad+ \frac{1}{3N} \sum_{\substack{d \mid N \\ \gcd(d, 6) = 2}} \mu(d) \left(3^{\frac{N}{d}} - 1 - \zeta_6^{-2\left(\frac{d}{2}\right)^{-1}} - \zeta_6^{-4\left(\frac{d}{2}\right)^{-1}}\right).
\end{align*}
For the summations over $d$ with $\gcd(d, 6) = 1$, $\frac{N}{d}$ is a multiple of 4, hence $(-\sqrt{-3})^{\frac{N}{d}} =  3^{\frac{N}{2d}} = (\sqrt{-3})^{\frac{N}{d}}$.
Also, we can easily check that $\zeta_6^{-d^{-1}} + \zeta_6^{-5d^{-1}} = \zeta_6 + \zeta_6^5 = 1$, $\zeta_6^{-2d^{-1}} + \zeta_6^{-4d^{-1}} = \zeta_6^2 + \zeta_6^4 = -1$. Hence the above expressions simplify as
\begin{align*}
    \pi(N; m, T+2) &= \frac{1}{6N} \sum_{\substack{d \mid N \\ \gcd(d, 6) = 1}} \mu(d)\left(3^{\frac{N}{d}} - 1 - (3^{\frac{N}{2d}} - 1)\right) = \frac{1}{6N}\sum_{\substack{d \mid N \\ \gcd(d, 6) = 1}} \mu(d) \left(3^{\frac{N}{d}} - 3^{\frac{N}{2d}}\right) \\
    \pi(N; m, (T+2)^2) &= \frac{1}{6N}\sum_{\substack{d \mid N \\ \gcd(d, 6) = 1}} \mu(d) \left(3^{\frac{N}{d}} - 1 - (-3^{\frac{N}{2d}} - 1)\right) + \frac{1}{3N}\sum_{\substack{d \mid N \\ \gcd(d, 6) = 2}} \mu(d)3^{\frac{N}{d}} \\
    &= \frac{1}{6N} \sum_{\substack{d \mid N \\ \gcd(d, 6) = 1}} \mu(d) \left(3^{\frac{N}{d}} + 3^{\frac{N}{2d}}\right) + \frac{1}{3N} \sum_{\substack{d \mid \frac{N}{2} \\ \gcd(d, 6) = 1}} \mu(2d) 3^{\frac{N}{2d}} \\
    &= \frac{1}{6N} \sum_{\substack{d \mid N \\ \gcd(d, 6) = 1}} \mu(d) \left(3^{\frac{N}{d}} + 3^{\frac{N}{2d}}\right) - \frac{1}{3N} \sum_{\substack{d \mid \frac{N}{2} \\ \gcd(d, 6) = 1}} \mu(d) 3^{\frac{N}{2d}} 
\end{align*}
Now $\pi(N; m, T+2) = \pi(N; m, (T+2)^2)$ follows by comparing the summations, where we have $d \mid N \Leftrightarrow d \mid \frac{N}{2}$ for odd $d$.

Table \ref{tab:p3T2} shows the number of primes over $\bF_3$ in each congruence class modulo $T^2$.

%% file: 4bijection.tex
\section{Explicit bijections}
\label{sec:bijection}

In this section, we provide an alternative approach to proving ties using \(\GL_2(\bF_q)\)-actions.

\subsection{Bijection via $\GL_2(\bF_q)$-action}

For each $c \in (A /m)^\times$, define $\widetilde{\cI}(N; m, c)$ to be the set of irreducible polynomials of degree $N$ satisfying $f \equiv c \pmod{m}$, \emph{not necessarily monic}, and let $\widetilde{\pi}(N; m, c) := \# \widetilde{\cI}(N; m, c)$ be the number of such polynomials.
This set clearly contains $\cI(N; m, c)$.
For $f = \sum_{i=0}^{n} a_i T^i \in \bF_q[T]$, write $\lc(f) := a_n$ for the leading coefficient of $f$.
By normalizing a polynomial, we have
\begin{equation}
    \label{eqn:congnorm}
    f \equiv c \pmod{m} \,\,\Leftrightarrow\,\, \lc(f)^{-1} f \equiv \lc(f)^{-1} c \pmod{m}
\end{equation}
which implies
\[
\#\{f \in \widetilde{\cI}(N; m, c) : \lc(f) = \lambda\} = \pi(N; m, \lambda^{-1} c).
\]
By summing over all possible leading coefficients, we get
\begin{equation}
    \label{eqn:Itildesum}
    \widetilde{\pi}(N; m, c) = \sum_{\lambda \in \bF_q^\times} \pi(N; m, \lambda^{-1} c) = \sum_{\lambda \in \bF_q^\times} \pi(N; m, \lambda c).
\end{equation}
From this, we have the following ``uninteresting`` ties between non-monic prime counts.
\begin{proposition}
    \label{eqn:tie_nonmonic_const}
    For any $c \in (\bF_q[T] / m)^\times$ and $\lambda \in \bF_q^\times$, we have
    \begin{equation}
        \widetilde{\pi}(N; m, c) = \widetilde{\pi}(N; m, \lambda c).
    \end{equation}
\end{proposition}

Now, for a matrix $B = \left(\begin{smallmatrix}
        \alpha & \beta \\ \gamma & \delta
    \end{smallmatrix}\right) \in \GL_2(\bF_q)$, define
\begin{equation}
    (f|_n B)(T) := (\gamma T + \delta)^n f\left(\frac{\alpha T + \beta}{\gamma T + \delta}\right) = \sum_{i=0}^{n} a_i (\alpha T + \beta)^i (\gamma T + \delta)^{n-i} \label{eqn:GL2action}
\end{equation}
For example, when $f(0) \ne 0$, $n = \deg f$ and $B = \left(\begin{smallmatrix}
    0 & 1 \\ 1 & 0 \end{smallmatrix}\right)$, then $f|_n B$ is a reciprocal of $f(T)$, i.e.
\[
f(T) = a_n T^n + a_{n-1} T^{n-1} + \cdots + a_1 T + a_0 \Rightarrow (f|_n B)(T) = a_0 T^n + a_1 T^{n-1} + \cdots + a_{n-1} T + a_n.
\]
For each $n$, the map $f \mapsto f|_n B$ defines an action of $\GL_2(\bF_q)$ on $\bF_q[T]$, i.e. $f|_{n}(B_1 B_2 ) = (f|_{n}B_1)|_{n}B_2$ for all $B_1, B_2 \in \GL_2(\bF_q)$.

\begin{proposition}
    \label{prop:gl2bijection_nonmono}
    Let $N \ge 2$ and $B = \left(\begin{smallmatrix}
        \alpha & \beta \\ \gamma & \delta
    \end{smallmatrix}\right) \in \GL_2(\bF_q)$.
    Assume that $m|_{M}B$ is a nonzero constant multiple of $m$ for $M = \deg m$.
    Then the map
    \begin{equation}
    \label{eqn:bij}
        f \mapsto f|_N B
    \end{equation}
    defines a bijection between $\widetilde{\cI}(N; m, c)$ and $\widetilde{\cI}(N; m, c|_{N} B)$.
    In particular, we have a tie
    \[
    \widetilde{\pi}(N; m, c) = \widetilde{\pi}(N; m, c|_N B).
    \]
\end{proposition}
\begin{proof}
    To show that the map preserves the degree, note that the leading coefficient of $f|_N B$ is
    \begin{equation}
        \label{eqn:lc}
        \sum_{i=0}^{N} a_i \alpha^i \gamma^{N-i} = \begin{cases}
            \gamma^N \sum_{i=0}^{N} a_i \left(\frac{\alpha}{\gamma}\right)^i = \gamma^N f\left(\frac{\alpha}{\gamma}\right) & \gamma \ne 0 \\
            a_N \alpha^N & \gamma = 0.
        \end{cases}
    \end{equation}
    When $\gamma \ne 0$, $f\left(\frac{\alpha}{\gamma}\right) \ne 0$ since $f$ is irreducible of degree $N \ge 2$ and cannot have a zero in $\bF_q$.
    When $\gamma$ is 0, $\alpha \ne 0$ (otherwise $B$ is not invertible) and $a_N \alpha^N \ne 0$.
    The map is a bijection, since the inverse map is given by
    \[
        f \mapsto f|_N B^{-1}.
    \]
    Regarding irreducibility, if $f = gh$ is reducible then $f|_n B = g|_{n_1} B \cdot h|_{n_2} B$ for $n_1 = \deg g$ and $n_2 = \deg h$, so $f|_n B$ is also reducible and vice versa.
    Finally, from $f \in \cI(N; m, c)$, one can write $f = gm + c$, and applying $|_N B$ on both sides gives
    \[
        f|_N B = (g|_{\deg g} B) \cdot (m|_{M} B) + c|_N B = (g|_{\deg g} B) \cdot \alpha' m + c|_N B
    \]
    for some nonzero constant $\alpha' \in \bF_q$, so $f|_N B \equiv c|_N B \pmod{m}$.
\end{proof}

Although the congruence class $c|_N B$ of the target $\cI(N; m, c|_{N} B)$ seems to depend on the degree $N$, it depends only on $N$ modulo a certain number (for fixed $B$).

\begin{lemma}
\label{lem:action_period}
    Let $m \in \bF_q[T]$ be a monic polynomial and $B = \left(\begin{smallmatrix} \alpha & \beta \\ \gamma & \delta \end{smallmatrix} \right) \in \GL_2(\bF_q)$ be a matrix such that $m|_{M} B = \lambda \cdot m$ for some $\lambda \in \bF_q^\times$ and $M = \deg m$.
    For a polynomial $c \in \bF_q[T]$ coprime to $m$, the sequence $N \mapsto c|_N B \pmod{m}$ is periodic in $N$ with the period equal to the order of $\gamma T + \delta$ in $(\bF_q[T] / m)^\times$.
\end{lemma}
\begin{proof}
    We have $(c|_{n+1} B)(T) = (\gamma T + \delta) (c|_{n} B)(T)$.
    Also, note that $c|_n B$ is coprime to $m$ for all $n$, which follows from Proposition \ref{prop:gl2bijection_nonmono}.
    Then we have $c|_{n + k} B = c|_{n} B$ for all $n$, where $k$ is the order of $\gamma T + \delta$ in the group $(\bF_q[T] / m)^\times$.
\end{proof}

\begin{corollary}
\label{cor:tie_periodic}
Let $m \in \bF_q[T]$ be a monic polynomial and $B = \left(\begin{smallmatrix}
    \alpha & \beta \\ \gamma & \delta
\end{smallmatrix}\right) \in \GL_2(\bF_q)$ be a matrix satisfying the same assumptions as in Lemma \ref{lem:action_period}, and let $N_0$ be the order of $\gamma T + \delta$ in $(\bF_q[T] / m)^\times$.
Fix $e \ge M - 1$ and let $c \in \bF_q[T]$ be a polynomial coprime to $m$ and $\deg c < M$.
For each $N \equiv e \pmod{N_0}$, we have a tie
\[
\widetilde{\pi}(N; m, c) = \widetilde{\pi}(N; m, c|_{e} B)
\]
induced by the bijection \eqref{eqn:bij}.
\end{corollary}

However, our main interest is in ties for irreducible \emph{monic} polynomials, and the map \eqref{eqn:bij} does not preserve monicity in general.
We are interested in the case where we can choose $B$ so that
\begin{equation}
    \label{eqn:monicdeg}
    \lc(f|_N B) = 1 \,\,\text{for all}\,f \in \cI(N; m, c)
\end{equation}
There are several cases in which \eqref{eqn:monicdeg} holds.

\begin{proposition}
    \label{prop:Bchoice1}
    For $N \ge 2$ and $B = \left(\begin{smallmatrix} \alpha & \beta \\ \gamma & \delta \end{smallmatrix}\right) \in \GL_2(\bF_q)$, \eqref{eqn:monicdeg} holds if
    \begin{enumerate}
        \item $q = 2$
        \item $\alpha^N = 1$ and $\gamma = 0$ (in particular, when $(\alpha, \gamma) = (1, 0)$)
    \end{enumerate}
\end{proposition}
\begin{proof}
    By \eqref{eqn:GL2action}, the coefficient of $T^N$ in $f|_N B$ is
    \begin{equation}
        \label{eqn:Ncoeff}
        \sum_{i=0}^{N} a_i \alpha^i \gamma^{N-i} = \begin{cases}
            \gamma^N \sum_{i=0}^{N} a_i \left(\frac{\alpha}{\gamma}\right)^i = \gamma^N f\left(\frac{\alpha}{\gamma}\right) & \text{if }\gamma \ne 0 \\
            a_N \alpha^N & \text{if }\gamma = 0
        \end{cases}
    \end{equation}
    When $q = 2$, note that the only nonzero element of $\bF_2$ is $1$.
    If $\gamma \ne 0$, $f(\frac{\alpha}{\gamma}) \ne 0$ since $f$ is irreducible of degree $\ge 2$, hence $\gamma^N f(\frac{\alpha}{\gamma}) = 1$.
    When $\gamma = 0$, the condition $B \in \GL_2(\bF_2)$ gives $\alpha = 1$, and the coefficient of $T^N$ in $f|_N B$ is again nonzero.

    For the second case, monicity is again preserved by \eqref{eqn:Ncoeff}.
\end{proof}

\subsection{Examples}

\subsubsection{\texorpdfstring{\(m = T^2 + T + 1 \in \bF_2[T]\)}{m = T^2 + T + 1 over characteristic 2}}
\label{subsec:bij_T2T1}

$m(T) = T^2 + T + 1$ is fixed under the action of any matrix in $\GL_2(\bF_2)$.
For example, it is fixed under the action of $B = \left(\begin{smallmatrix} 0 & 1 \\ 1 & 0 \end{smallmatrix} \right) \in \GL_2(\bF_2)$, i.e. $T^2 m\left(\frac{1}{T}\right) = m(T)$, which reflects the fact that $m(T)$ is palindromic.
$\gamma T + \delta = T$ has order $3$ in $(\bF_2[T] / m)^\times \simeq \bZ/3$, which means that ties occur in the prime race modulo $m(T)$ for each degree class modulo $3$.
In particular, the action sends
\begin{align*}
    N \equiv 0 \pmod{3} &: 1 \mapsto 1, \quad T \mapsto T + 1, \quad T+1 \mapsto T \\
    N \equiv 1 \pmod{3} &: 1 \mapsto T, \quad T \mapsto 1, \quad T + 1 \mapsto T + 1 \\
    N \equiv 2 \pmod{3} &: 1 \mapsto T + 1, \quad T \mapsto T, \quad T + 1 \mapsto 1
\end{align*}
and this yields the ties
\begin{align*}
    N \equiv 0 \pmod{3} &: \pi(N; m, T) = \pi(N; m, T+1) \\
    N \equiv 1 \pmod{3} &: \pi(N; m, 1) = \pi(N; m, T) \\
    N \equiv 2 \pmod{3} &: \pi(N; m, 1) = \pi(N; m, T+1).
\end{align*}
If we consider the translation matrix $B = \left(\begin{smallmatrix} 1 &1 \\ 0 & 1 \end{smallmatrix}\right)$, it yields $\pi(N; m, T) = \pi(N; m, T + 1)$ for \emph{all} $N$.
Combining these, we get the same ties as in Section \ref{subsubsec:p2T2T1_explicit} (Table \ref{tab:T2T1group}).

More generally, any palindromic polynomial $m(T) \in \bF_2[T]$ is fixed under the action of $B = \left(\begin{smallmatrix} 0 & 1 \\ 1 & 0 \end{smallmatrix}\right)$, and it induces a tie $\pi(N; m, c) = \pi(N; m, c|_e B)$ for any $c \in (\bF_2[T] / m)^\times$ for all $N \equiv e \pmod{N_0}$, where $N_0$ is the order of $T$ in $(\bF_2[T] / m)^\times$.

\subsubsection{\texorpdfstring{\(m = T^3 + T + 1 \in \bF_2[T]\)}{m = T^3 + T + 1 over characteristic 2}}
\label{subsec:bij_T3T1}

Consider the matrix $B = \left(\begin{smallmatrix} 1 & 1 \\ 1 & 0 \end{smallmatrix} \right) \in \GL_2(\bF_2)$.
Then $m(T) = T^3 + T + 1$ is fixed under the action of $B$, i.e. $T^2 m\left(\frac{1}{T}+1\right) = m(T)$, and $\gamma T + \delta = T$ has order $7$ in $(\bF_2[T]/m)^{\times}$.
By Corollary \ref{cor:tie_periodic}, we have ties $\pi(N; m, c) = \pi(N; m, c|_e B)$ for each $e \in \{0, 1, \dots, 6\}$ and $N \equiv e \pmod{7}$.
For example, when $e = 1$, the $\GL_2$-action sends 
\[
1 \mapsto T,\quad 
T \mapsto T + 1,\quad 
T + 1 \mapsto 1,
\]
and
\[
T^2 \mapsto T^2 + T,\quad 
T^2 + T \mapsto T^2 + T + 1,\quad 
T^2 + T + 1 \mapsto T^2,
\]
which gives the ties
\begin{align*}
    \pi(N; m, 1) &= \pi(N; m, T) = \pi(N; m, T+1),\\
    \pi(N; m, T^2) &= \pi(N; m, T^2+T) = \pi(N; m, T^2+T+1).
\end{align*}
Similarly, we obtain the ties for the other residue classes modulo $7$, which agree with those in Section \ref{subsubsec:p2T3T1} (Table \ref{tab:T3T1}).

\subsubsection{\texorpdfstring{\(m = T^2 + 1 \in \bF_3[T]\)}{m = T^2 + 1 over characteristic 3}}
\label{subsubsec:p3T21}

Observe that $m$ is irreducible in $\mathbb{F}_3[T]$. Let $B = \left(\begin{smallmatrix} 1 & 0 \\ 0 & 2 \end{smallmatrix} \right)$. Then $(m|_2 B)(T) = m(T)$ and $\gamma T + \delta = 2$ has order $2$ in $\left(\mathbb{F}_3[T]/(T^2 + 1)\right)^\times$. We thus consider degrees modulo $2$; the action sends each congruence class as
\begin{align*}
    N \equiv 0 \pmod{2} : \quad & 1 \mapsto 1, \quad 2 \mapsto 2, \quad T \mapsto 2T, \quad T + 1 \mapsto 2T + 1, \quad T + 2 \mapsto 2T + 2, \\ 
    & 2T \mapsto T, \quad 2T + 1 \mapsto T + 1, \quad 2T + 2 \mapsto T + 2; \\
    N \equiv 1 \pmod{2} : \quad & 1 \mapsto 2, \quad 2 \mapsto 1, \quad T \mapsto T, \quad T + 1 \mapsto T + 2, \quad T + 2 \mapsto T + 1, \\ 
    & 2T \mapsto 2T, \quad 2T + 1 \mapsto 2T + 2, \quad 2T + 2 \mapsto 2T + 1;
\end{align*}
and this gives the same ties as in Section \ref{subsubsec:p3T21explicit} (Table \ref{tab:p3T21group}).

\subsubsection{\texorpdfstring{\(m = T^2 \in \bF_p[T]\)}{m = T^2 over characteristic p}}
\label{subsubsec:T2}

Let $m(T) = T^2 \in \bF_p[T]$.
When $p = 2$, one can take $B = \left(\begin{smallmatrix}
    1 & 0 \\ 1 & 1
\end{smallmatrix}\right)$ so that for odd $N$ and $c(T) = 1$, we have
\[
(c|_{N}B)(T) = (T + 1)^N \equiv NT + 1 \equiv T + 1 \pmod{T^2},
\]
and we get the tie $\pi(N; m, 1) = \pi(N; m, T + 1)$ (Table \ref{tab:p2T2}).

For odd $p$, we cannot use the same $B$ since the map $f \mapsto f|_{N}B$ does not preserve monicity in general (even if $m|_2 B = m$).
Instead, we choose $B = \left(\begin{smallmatrix} 1 & 0 \\ 0 & \delta\end{smallmatrix}\right)$, and for $c(T) = c_1 T + c_0 \in (\bF_p[T] / m)^\times$,
\begin{equation}
    (c|_N B)(T) = \delta^N \left(c_1 \frac{T}{\delta} + c_0\right) = \delta^{N-1} c_1 T + \delta^N c_0.
\end{equation}
If the order of $\delta \in \bF_p^\times$ is $N_0$, then for each $e \ge 1$ coprime to $p - 1$, the action of $B$ gives ties $\pi(N; m, c) = \pi(N; m, c|_e B)$ for $N \equiv e \pmod{N_0}$.
When $\delta$ is a generator of the multiplicative group $\bF_p^\times$ (so has order $N_0 = p - 1$), $\delta^e$ also has order $p - 1$ and $(c, c|_e B, c|_e B^2, \dots, c|_e B^{p-2})$ forms a cycle of length $p-1$ that induces ties
\[
\pi(N; m, c_1 T + c_0) = \pi(N; m, \delta^{e - 1} c_1 T + \delta^e c_0) = \cdots = \pi(N; m, \delta^{(p-2)(e - 1))}c_1 T + \delta^{(p-2)e} c_0)
\]
In particular, when $e = 1$, it shows that $\pi(N; m, c_1 T + c_0)$ is independent of $c_0 \ne 0$ for $N \equiv 1 \pmod{p - 1}$.

For example, when $p = 3$ and $B = \left(\begin{smallmatrix} 1 & 0 \\ 0 & 2 \end{smallmatrix}\right)$, $\delta = 2 \in \bF_3^\times$ has order $2$, and it gives the ties
\begin{align*}
    \pi(N; m, 1) = \pi(N; m, 2), \quad \pi(N; m, T + 1) = \pi(N; m, T + 2), \quad \pi(N; m, 2T + 1) = \pi(N; m, 2T + 2)
\end{align*}
for odd $N$.
For even $N$, the same $B$ gives ties between $c_1 T + c_0$ and $2c_1 T + c_0$, i.e.
\begin{align*}
    \pi(N; m, T + 1) = \pi(N; m, 2T + 1), \quad \pi(N; m, T + 2) = \pi(N; m, 2T + 2).
\end{align*}

\subsubsection{\texorpdfstring{\(m = T^p - T - 1 \in \bF_p[T]\)}{m = T^p - T - 1 over characteristic p}}

Artin--Schreier polynomials $m(T) = T^p - T - a \in \bF_p[T]$ are irreducible for any $a \in \bF_p^\times$.
The polynomial is fixed under the action of $B = \left(\begin{smallmatrix} 1 & 1 \\ 0 & 1 \end{smallmatrix}\right)$, i.e. translation $T \mapsto T + 1$.
Hence it induces ties $\pi(N; m, c_1) = \pi(N; m, c_2)$ for $c_1, c_2 \in (\bF_p[T] / m)^\times$ whenever
\[
c_2(T) = c_1(T + b)\quad\text{for some}\,\,b \in \bF_p.
\]
For example, when $p = 3$, $m(T) = T^3 - T - 1 = T^3 + 2T + 2$ and $N = 24$, we have
\[
\pi(24; m, T^2) = \pi(24; m, T^2 + 2T + 1) = \pi(24; m, T^2 + T + 1) = 452605575.
\]

%% file: 5conclusion.tex
\section{Conclusion}

In this paper, we proposed two approaches for proving ties in prime counts over function fields: one via explicit formulas and one via bijections arising from \(\GL_2\)-actions.
A natural question is how these two approaches are related, and whether one method is always stronger than the other.
For example, in Section \ref{subsubsec:T2} we showed that the ties
\[
\pi(N; m, T + 1) = \pi(N; m, T + 2) = \pi(N; m, 2T + 1) = \pi(N; m, 2T + 2)
\]
hold for $m(T) = T^2 \in \mathbb{F}_3[T]$ and $N \equiv 0 \pmod{4}$, but the bijection method can only prove half of these ties, and we could not find an appropriate matrix that yields all of them. However, we also found that it is often easier to prove ties by finding a suitable $B \in \GL_2(\bF_q)$ that gives the desired bijection than by finding exceptional Galois conjugates.

One may also ask whether there are interesting ties for the cumulative version of prime counts considered in \cite{cha2008chebyshev} and the subsequent literature.
Although there are some examples with small $N$, we conjecture that such ties are very rare and almost never happen.
\begin{conjecture}
    Let $m \in \mathbb{F}_q[T]$ and let $a, b \in (\mathbb{F}_q[T]/m)^\times$ be two distinct congruence classes modulo $m$.
    Then there are finitely many $N$ such that
    \[
        \sum_{n=1}^{N} \pi(n; m, a) = \sum_{n=1}^{N} \pi(n; m, b).
    \]
\end{conjecture}
For example, Table \ref{tab:T3T1cum} shows the cumulative prime counts for $m = T^3 + T + 1 \in \bF_2[T]$, where no ties appear for $22 \le N \le 40$.

In Remark \ref{rmk:T2T1bias}, we also proved that the primes are always biased against $c = 1$ modulo $T^2 + T + 1 \in \bF_2[T]$, that is, the inequality $\pi(N; m, 1) < \pi(N; m, T)$ holds whenever $N \ge 9$ is a multiple of $3$.
It would be interesting to prove this by constructing an \emph{injection} from $\cI_2(N; m, 1)$ to $\cI_2(N; m, T)$.

%% file: appendix.tex
\appendix

\section{Tables}

\begin{table}[h]
\centering
\begin{tabular}{c|ccc}
\toprule
{$N$} & {$1$} & {$T$} & {$T^2$} \\
\midrule
10 & \myg{33} & \myg{33} & \myg{33} \\
11 & \myg{62} & \myg{62} & \myg{62} \\
12 & 111      & \myb{112} & \myb{112} \\
13 & \myg{210} & \myg{210} & \myg{210} \\
14 & \myg{387} & \myg{387} & \myg{387} \\
15 & 726      & \myb{728} & \myb{728} \\
16 & \myg{1360} & \myg{1360} & \myg{1360} \\
17 & \myg{2570} & \myg{2570} & \myg{2570} \\
18 & 4842      & \myb{4845} & \myb{4845} \\
19 & \myg{9198} & \myg{9198} & \myg{9198} \\
20 & \myg{17459} & \myg{17459} & \myg{17459} \\
\midrule
$0 \pmod{3}$ &   & \myb{$\circ$}    & \myb{$\circ$}    \\
$1 \pmod{3}$ & \myg{\ding{73}}   & \myg{\ding{73}}    & \myg{\ding{73}}    \\
$2 \pmod{3}$ & \myg{\ding{73}}   & \myg{\ding{73}}   & \myg{\ding{73}}     \\
\bottomrule
\end{tabular}%
\caption{Ties in the number of primes modulo \(T^2 + T + 1 \in \mathbb{F}_2[T]\) for degrees from 10 to 20, together with the corresponding patterns modulo 3.}
\label{tab:T2T1group}
\end{table}


\begin{table}[h]
\centering
\resizebox{\columnwidth}{!}{%
\begin{tabular}{c|cccccccc}
\toprule
{$N\pmod{2}$} & {$1$} & {$T+1$} & {$(T+1)^2$} & {$(T+1)^3$} & {$(T+1)^4$} & {$(T+1)^5$} & {$(T+1)^6$} & {$(T+1)^7$} \\
\midrule
10 & 720 & \myb{737} & \myg{732} & \myb{737} & 744 & \myr{739} & \myg{732} & \myr{739} \\
11 & \myb{2013} & \myg{2025} & 2004 & \myr{2001} & \myb{2013} & \myr{2001} & 2022 & \myg{2025} \\
12 & 5506 & \myb{5554} & \myg{5520} & \myb{5554} & 5534 & \myr{5516} & \myg{5520} & \myr{5516} \\
13 & \myb{15330} & \myg{15325} & 15282 & \myr{15335} & \myb{15330} & \myr{15335} & 15378 & \myg{15325} \\
14 & 42720 & \myb{42745} & \myg{42666} & \myb{42745} & 42612 & \myr{42665} & \myg{42666} & \myr{42665} \\
15 & \myb{119572} & \myg{119484} & 119548 & \myr{119660} & \myb{119572} & \myr{119660} & 119596 & \myg{119484} \\
16 & 336387 & \myb{336243} & \myg{336200} & \myb{336243} & 336013 & \myr{336362} & \myg{336200} & \myr{336362} \\
17 & \myb{949560} & \myg{949440} & 949848 & \myr{949680} & \myb{949560} & \myr{949680} & 949272 & \myg{949440} \\
18 & 2690096 & \myb{2690030} & \myg{2690142} & \myb{2690030} & 2690188 & \myr{2690800} & \myg{2690142} & \myr{2690800} \\
19 & \myb{7646457} & \myg{7646865} & 7647144 & \myr{7646049} & \myb{7646457} & \myr{7646049} & 7645770 & \myg{7646865} \\
20 & 21790236 & \myb{21792137} & \myg{21791664} & \myb{21792137} & 21793092 & \myr{21792667} & \myg{21791664} & \myr{21792667} \\
\midrule
$0 \pmod{2}$ &  & \myb{$\circ$} & \myg{\ding{73}} & \myb{$\circ$} & & \myr{$\triangle$} & \myg{\ding{73}} & \myr{$\triangle$}   \\
$1 \pmod{2}$ & \myb{$\circ$} & \myg{\ding{73}} & & \myr{$\triangle$} & \myb{$\circ$} & \myr{$\triangle$} & & \myg{\ding{73}} \\
\bottomrule
\end{tabular}%
}
\caption{Ties in the number of primes modulo \(T^2 + 1 \in \mathbb{F}_3[T]\) for degrees from 10 to 20, together with the corresponding patterns modulo 2.}
\label{tab:p3T21group}
\end{table}

\begin{table}[h]
\centering
\begin{tabular}{c|cc}
\toprule
{$N$} & {$1$} & {$T+1$}  \\
\midrule
10 & 48 & 51 \\ 
11 & \myb{93} & \myb{93} \\
12 & 165 & 170 \\
13 & \myb{315} & \myb{315} \\
14 & 576 & 585 \\ 
15 & \myb{1091} & \myb{1091} \\ 
16 & 2032 & 2048 \\ 
17 & \myb{3855} & \myb{3855} \\ 
18 & 7252 & 7280 \\
19 & \myb{13797} & \myb{13797} \\
20 & 26163 & 26214 \\ 
\midrule
$0 \pmod{2}$ & & \\
$1 \pmod{2}$ & $\myb{\circ}$ & $\myb{\circ}$ \\
\bottomrule
\end{tabular}%
\caption{Number of irreducible polynomials modulo \(T^2 \in \mathbb{F}_2[T]\) in the two nontrivial congruence classes for each degree $N$.}
\label{tab:p2T2}
\end{table}

\begin{table}[h]
\centering
\begin{tabular}{c|cccccc}
\toprule
$N$ & $1$ & $T+2$ & $(T+2)^2$ & $(T+2)^3$ & $(T+2)^4$ & $(T+2)^5$   \\
\midrule
 10  & 984 & \myr{988} & \myb{972} & 976 & \myb{972} & \myr{988} \\
 11  & \myr{2684} & \myb{2673}  & \myb{2673} & \myr{2684} & \myg{2695} & \myg{2695} \\
 12  & 7338 & \myb{7371} & \myb{7371}  & 7398 & \myb{7371} & \myb{7371}  \\
 13  & \myr{20440} & \myb{20468} & \myb{20468} & \myr{20440} & \myg{20412} & \myg{20412} \\
 14  & 56940 & \myr{56966} & \myb{56862} & 56888 & \myb{56862} & \myr{56966} \\
 15  & \myr{159424} & \myb{159359} & \myb{159359} & \myr{159424} & \myg{159505} & \myg{159505} \\
 16  & 448130 & \myb{448335} & \myb{448335} & 448540 & \myb{448335} & \myb{448335} \\
 17  & \myr{1266080} & \myb{1266273} & \myb{1266273} & \myr{1266080} & \myg{1265887} & \myg{1265887}  \\
 18  & 3587208 & \myr{3587409} & \myb{3586680} & 3586842 & \myb{3586680} & \myr{3587409} \\
 19  & \myr{10195276} & \myb{10194758} & \myb{10194758} & \myr{10195276} & \myg{10195794} & \myg{10195794} \\
 20  & 29054568 & \myb{29056044} & \myb{29056044} & 29057520 & \myb{29056044} & \myb{29056044} \\
\midrule
$0 \pmod{4}$ & & $\myb{\circ}$ & $\myb{\circ}$ & & $\myb{\circ}$ & $\myb{\circ}$ \\
$1 \pmod{4}$ & $\myr{\triangle}$ & $\myb{\circ}$ & $\myb{\circ}$ & $\myr{\triangle}$ & \myg{\ding{73}} & \myg{\ding{73}} \\
$2 \pmod{4}$ & & $\myr{\triangle}$ & $\myb{\circ}$ & & $\myb{\circ}$ & $\myr{\triangle}$ \\
$3 \pmod{4}$ & $\myr{\triangle}$ & $\myb{\circ}$ & $\myb{\circ}$ & $\myr{\triangle}$ & \myg{\ding{73}} & \myg{\ding{73}} \\
\bottomrule
\end{tabular}%
\caption{Number of irreducible polynomials modulo \(T^2 \in \mathbb{F}_3[T]\) in the six nontrivial congruence classes for each degree $N$.}
\label{tab:p3T2}
\end{table}

\begin{table}[h!]
\centering
\begin{tabular}{c|cccccccc}
\toprule
\textbf{$N$} & \textbf{$1$} & \textbf{$T$} & \textbf{$T^2$} & \textbf{$T+1$} & \textbf{$T^2+T$} & \textbf{$T^2+T+1$} & \textbf{$T^2+1$} \\
\midrule
1 & \myb{0} & \myg{1} & \myb{0} & \myg{1} & \myb{0} & \myb{0} & \myb{0} \\
2 & \myb{0} & \myg{1} & \myb{0} & \myg{1} & \myb{0} & \myg{1} & \myb{0} \\
3 & \myb{0} & \myg{1} & \myb{0} & \myg{1} & \myg{1} & \myg{1} & \myb{0} \\
4 & 0 & 2 & \myb{1} & \myb{1} & \myb{1} & \myb{1} & \myb{1} \\
5 & \myb{1} & 3 & \myb{1} & \myg{2} & \myg{2} & \myg{2} & \myg{2} \\
6 & 2 & \myb{3} & \myb{3} & \myg{4} & \myb{3} & \myg{4} & \myb{3} \\
7 & \myb{5} & \myg{6} & \myg{6} & \myg{6} & \myg{6} & \myg{6} & \myb{5} \\
8 & 9 & \myb{10} & \myb{10} & \myb{10} & \myb{10} & \myb{10} & 11 \\
9 & 16 & \myb{19} & 17 & \myb{19} & \myb{19} & \myg{18} & \myg{18} \\
10 & \myb{31} & 33 & \myg{32} & 34 & \myb{31} & \myg{32} & \myg{32} \\
11 & \myb{59} & \myb{59} & \myg{58} & \myg{58} & \myb{59} & 60 & \myg{58} \\
12 & \myb{105} & \myb{105} & \myg{108} & \myr{107} & \myg{108} & 106 & \myr{107} \\
13 & 194 & 201 & \myb{197} & \myg{196} & \myb{197} & 195 & \myg{196} \\
14 & 362 & 363 & \myb{359} & \myg{365} & \myb{359} & 364 & \myg{365} \\
15 & 672 & 673 & \myb{675} & \myb{675} & \myb{675} & 680 & 669 \\
16 & 1260 & 1255 & 1263 & \myb{1257} & \myb{1257} & 1250 & \myb{1257} \\
17 & 2353 & 2364 & 2356 & 2350 & 2361 & 2359 & 2366 \\
18 & 4428 & 4433 & 4425 & 4450 & 4436 & 4434 & 4435 \\
19 & 8379 & 8384 & \myb{8385} & 8377 & 8363 & \myb{8385} & 8362 \\
20 & 15881 & 15842 & \myb{15856} & 15848 & 15865 & \myb{15856} & 15864 \\
21 & 30089 & \myb{30124} & 30138 & 30116 & 30147 & \myb{30124} & 30132 \\
22 & 57347 & 57382 & 57327 & 57374 & 57336 & 57313 & 57348 \\
23 & 109455 & 109448 & 109435 & 109440 & 109402 & 109513 & 109456 \\
24 & 209321 & 209224 & 209301 & 209306 & 209346 & 209289 & 209232 \\
25 & 400992 & 401057 & 401134 & 400970 & 401017 & 400960 & 401065 \\
26 & 769719 & 769784 & 769546 & 769704 & 769751 & 769687 & 769799 \\
27 & 1479730 & 1479808 & 1479863 & 1480021 & 1479762 & 1480004 & 1479810 \\
28 & 2849599 & 2849372 & 2849427 & 2849299 & 2849326 & 2849282 & 2849088 \\
29 & 5494035 & 5493808 & 5494161 & 5493735 & 5494060 & 5494016 & 5494368 \\
30 & 10606547 & 10607133 & 10606673 & 10607060 & 10607385 & 10606772 & 10606880 \\
31 & 20503110 & 20503464 & 20503236 & 20503623 & 20502369 & 20503103 & 20503211 \\
32 & 39677586 & 39676638 & 39676410 & 39676353 & 39676845 & 39677579 & 39676385 \\
33 & 76862350 & 76861402 & 76863706 & 76862819 & 76863311 & 76862343 & 76862851 \\
34 & 149046001 & 149048350 & 149046544 & 149045657 & 149046962 & 149045181 & 149046502 \\
35 & 289290199 & 289290123 & 289288317 & 289291420 & 289288735 & 289290944 & 289292265 \\
36 & 561985241 & 561985165 & 561985719 & 561986462 & 561986137 & 561988346 & 561981893 \\
37 & 1092641229 & 1092635899 & 1092641707 & 1092637196 & 1092636871 & 1092634490 & 1092637881 \\
38 & 2126009274 & 2126013541 & 2126009752 & 2126005241 & 2126015143 & 2126012132 & 2126015523 \\
39 & 4139763776 & 4139768673 & 4139764884 & 4139779049 & 4139769645 & 4139766634 & 4139770655 \\
40 & 8066595506 & 8066600403 & 8066595408 & 8066591969 & 8066582565 & 8066598364 & 8066583575 \\
\bottomrule
\end{tabular}%
\caption{Cumulative count of the number of primes modulo \(T^3 + T + 1 \in \mathbb{F}_2[T]\) in each of the seven congruence classes up to degree $N$. The columns are ordered as \(1, T, T^2, T^3, \dots, T^6\).}
\label{tab:T3T1cum}
\end{table}